\documentclass[a4paper,11pt,oneside]{amsart}
\usepackage{amsmath,amsfonts,amssymb,amsthm,amscd,latexsym,euscript}

\usepackage{pdfsync}

\newtheorem{theorem}{Theorem}[section]
\newtheorem{corollary}[theorem]{Corollary}
\newtheorem{definition}[theorem]{Definition}

\newtheorem{lemma}[theorem]{Lemma}
\newtheorem{proposition}[theorem]{Proposition}

\newcommand{\Ce}{{\mathcal C}}

\numberwithin{equation}{section}
\begin{document}

\title{$C^{(n)}$-cardinals}
\author{Joan Bagaria}
\thanks{The author was supported by the Spanish 
Ministry of Education and Science under grant MTM2008-03389/MTM,
and by the Generalitat de Catalunya under grant  2009SGR-00187. Part of this work was done in November 2009 during the author's stay at the Mittag-Leffler Institut, whose support and hospitality is gratefully acknowledged.}
\date{\today }

\begin{abstract}
For each natural number $n$, let $C^{(n)}$ be the closed and unbounded proper class of ordinals $\alpha$ such that $V_\alpha$ is a $\Sigma_n$ elementary substructure of $V$. We say that $\kappa$ is a  \emph{$C^{(n)}$-cardinal} if it is the critical point of an elementary embedding $j:V\to M$, $M$ transitive, with $j(\kappa)$ in $C^{(n)}$. By analyzing the notion of $C^{(n)}$-cardinal  at various levels of the usual hierarchy of large cardinal principles we show that, starting at the level of superstrong cardinals and up to the level of rank-into-rank embeddings, $C^{(n)}$-cardinals form a much finer hierarchy. The naturalness of the notion of $C^{(n)}$-cardinal is exemplified by showing that the existence of $C^{(n)}$-extendible cardinals is equivalent to simple reflection principles for  classes of structures, which generalize the notions of supercompact and extendible cardinals. Moreover, building on results of \cite{BCMR}, we  give new characterizations of Vope\v{n}ka's Principle in terms of $C^{(n)}$-extendible cardinals.
\end{abstract}
\maketitle

\section{Introduction}


For each natural number $n$, let $C^{(n)}$ denote the \emph{club} (i.e., closed and unbounded\footnote{For all standard set-theoretic undefined notions, see \cite{J2}.}) proper class of ordinals $\alpha$ that are \emph{$\Sigma_n$-correct} in the universe $V$ of all sets, meaning that $V_{\alpha}$ is a $\Sigma_n$-elementary substructure of $V$, and written $V_{\alpha}\preceq_n V$. Observe that $\alpha$ is $\Sigma_n$-correct in $V$ if and only if it is \emph{$\Pi_n$-correct} in $V$, i.e., $V_\alpha$ is a $\Pi_n$-elementary substructure of $V$. Notice also that if $\alpha$ is $\Sigma_n$-correct and $\varphi$ is a $\Sigma_{n+1}$ sentence with parameters in $V_\alpha$ that holds in $V_\alpha$, then $\varphi$ holds in $V$. And if $\psi$ is a $\Pi_{n+1}$ sentence with parameters in $V_\alpha$ that holds in $V$, then it  holds in $V_\alpha$. These basic facts will be used throughout the paper without further comment.

The class  $C^{(0)}$ is the class of all ordinals. But  if  $V_{\alpha}\preceq_1 V$, then $\alpha$ is already an uncountable strong limit cardinal: clearly $\alpha$ is a limit cardinal greater than $\omega$, and if $\beta <\alpha$, then the sentence
$$\exists \gamma \exists f (\gamma \mbox{ an ordinal} \wedge f:\gamma \to V_{\beta} \mbox{ is onto})$$
is $\Sigma_1$ in the parameter $V_{\beta }$, and therefore it must hold in $V_{\alpha}$. Further, if $\alpha \in C^{(1)}$, then $V_{\alpha}=H_\alpha$. Thus, since $H_\alpha \preceq_1 V$ for every uncountable cardinal $\alpha$,  $C^{(1)}$ is precisely the class of all uncountable cardinals $\alpha$ such that $V_{\alpha}=H_\alpha$. It follows that $C^{(1)}$ is $\Pi_1$ definable, for $\alpha \in C^{(1)}$ if and only if  $\alpha$ is an uncountable  cardinal and
$$\forall M (M \mbox{ a transitive model of ZFC}^{\ast}\wedge \alpha \in M\to M\models V_{\alpha}=H_\alpha).$$
(Here $\mbox{ZFC}^{\ast}$ denotes a sufficiently large finite fragment of ZFC.)
The point is that if $\alpha \in C^{(1)}$ and $M$ is a transitive model of $\mbox{ZFC}^{\ast}$ that contains $\alpha$, then if in $M$ we could find some transitive $x\in V_{\alpha}\setminus H_\alpha$, we would have that $|x| \geq \alpha$. But this contradicts the fact that in $V$ the cardinality of $x$ is less than  $\alpha$, because $V_{\alpha}= H_\alpha$.

 More generally, since the truth predicate $\models_n$ for $\Sigma_n$ sentences (for $n\geq 1$) is $\Sigma_n$ definable  (see \cite{K}, Section 0.2),   and since the relation $x=V_y$ is $\Pi_1$, the class $C^{(n)}$, for $n\geq 1$, is $\Pi_{n}$ definable:  $\alpha \in C^{(n)}$  if and only if   
$$\alpha\in C^{(n-1)}\wedge \forall \varphi(x)\in \Sigma_{n} \forall a\in V_{\alpha}(\models_{n}\varphi (a) \to V_{\alpha} \models \varphi (a)).$$

Let us remark that $C^{(n)}$, for $n\geq 1$, cannot be $\Sigma_n$ definable. Otherwise, if $\alpha$ is the least ordinal in $C^{(n)}$, then the sentence  ``there is some ordinal in $C^{(n)}$" would be  $\Sigma_n$ and so it would  hold in $V_{\alpha}$, yielding an ordinal in $C^{(n)}$ smaller than $\alpha$, which is impossible. 

The classes $C^{(n)}$, $n \geq 1$,  form a \emph{basis} for definable club proper classes of ordinals, in the sense that  every $\Sigma_n$ club proper class of ordinals contains $C^{(n)}$. For suppose $\Ce$ is a club proper class of ordinals that is $\Sigma_n$ definable, some $n\geq 1$. If $\alpha \in C^{(n)}$, then for every $\beta <\alpha$, the sentence
$$\exists \gamma (\beta <\gamma \wedge \gamma \in \Ce)$$
is $\Sigma_n$ in the parameter $\beta$ and is true in $V$, hence also true in $V_\alpha$. This shows that $\Ce$ is unbounded below $\alpha$. Therefore, since $\Ce$ is closed,  $\alpha \in \Ce$.

By a similar argument one can show that every club proper class $C$ of ordinals that is $\mathbf{\Sigma_n}$ (i.e., $\Sigma_n$ definable with parameters)  contains all $\alpha \in C^{(n)}$ that are greater than the rank of the parameters involved in any given $\Sigma_n$ definition of $C$.

Finally, note that since the least ordinal in $C^{(n)}$ does not belong to $C^{(n+1)}$, $C^{(n+1)}\subset C^{(n)}$, all $n$.

\medskip

When working with non-trivial elementary embeddings $j:V\to M$, with $M$ transitive, one would like to have some control over where the image $j(\kappa)$ of the critical point $\kappa$ goes. An especially interesting case is when one wants $V_{j(\kappa)}$ to reflect some specific property of $V$ or, more generally, when one wants $j(\kappa)$ to belong to a particular definable club proper class of ordinals. Now, since the $C^{(n)}$, $n\in \omega$, form a basis for such classes, the problem can be reformulated as follows: when can one  
have $j(\kappa)\in C^{(n)}$, for a given $n\in \omega$? This prompts the following definition.

Let us say that a cardinal $\kappa$  is \emph{$C^{(n)}$-measurable} if there is an elementary embedding $j:V\to M$, some transitive class $M$,  with critical point $crit(j)=\kappa$ and with $j(\kappa)\in C^{(n)}$.

Observe  that if $j:V\to M\cong Ult(V,\mathcal{U})$, $M$ transitive, is the ultrapower elementary  embedding obtained from a non-principal $\kappa$-complete ultrafilter $\mathcal{U}$ on $\kappa$, then $2^{\kappa}<j(\kappa)<(2^{\kappa})^+$ (see \cite{K}). Hence, since $V_{j(\kappa)}\preceq_1 V$ implies that $j(\kappa)$ is a (strong limit) cardinal, $j$ cannot witness the $C^{(1)}$-measurability of $\kappa$.
Nonetheless, by using iterated ultrapowers (see \cite{J2}, 19.15), one has that for every cardinal $\alpha >2^{\kappa}$, the $\alpha$-th iterated ultrapower embedding 
$j_{\alpha}:V\to M_{\alpha}\cong Ult(V,\mathcal{U}_{\alpha})$, where $\mathcal{U}_{\alpha}$ is the $\alpha$-th iterate of $\mathcal{U}$, has critical point $\kappa$ and $j_{\alpha}(\kappa)=\alpha$.
So, if $\kappa$ is measurable, then for each $n$ one can always find an elementary embedding $j:V\to M$, $M$ transitive, with $j(\kappa)\in C^{(n)}$. We have thus shown the following.

\begin{proposition}
\label{prop1}
Every measurable cardinal  is   $C^{(n)}$-measurable, for all $n$.
\end{proposition}

A similar situation occurs in  the case of strong cardinals.

Let us say that  a cardinal $\kappa$ is \emph{$C^{(n)}$-strong} if for every $\lambda >\kappa$, $\kappa$ is \emph{$\lambda$-$C^{(n)}$-strong}, that is, there exists an elementary embedding $j:V\to M$, $M$ transitive, with critical point $\kappa$, and such that $j(\kappa)>\lambda$, $V_{\lambda}\subseteq M$, and $j(\kappa)\in C^{(n)}$.
Equivalently (see \cite{K}, 26.7), $\kappa$ is $\lambda$-$C^{(n)}$-strong if and only if  there exists  a  $(\kappa ,\beta)$-extender $E$, for some $\beta >|V_{\lambda}|$,  with $V_{\lambda}\subseteq M_E$ and $\lambda <j_E(\kappa)\in C^{(n)}$. 



Suppose now that $j:V\to M$ witnesses  the $\lambda$-strongness of $\kappa$, with $j(\kappa)$ not necessarily in $C^{(n)}$. Let $E$ be the $(\kappa , j(\kappa))$-extender obtained from $j$, and let $j_E:V\to M_E$ be the corresponding $\lambda$-strong embedding (see \cite{K}). Then in $M_E$, $E':= j_E(E)$ is a $(j_E(\kappa), j_E(j(\kappa)))$-extender, which gives rise to an elementary embedding $j_{E'}:M_E\to M_{E'}$ with critical point $j_E(\kappa)$. Still in $M_E$, let $\mathcal{U}$ be the standard $j_E(\kappa)$-complete ultrafilter on $j_E(\kappa)$ derived from $j_{E'}$, i.e., $$\mathcal{U}=\{ X\subseteq j_E(\kappa): j_{E}(\kappa)\in j_{E'}(X)\}$$ and let $j_{\mathcal{U}}:M_E\to M$ be the corresponding elementary embedding.
Then one can iterate $j_{\mathcal{U}}$  $\alpha$-times, for some $\alpha\in C^{(n)}$ greater than $2^{j_E(\kappa)}$, so that if $j_{\alpha}:M_E\to M_{\alpha}$ is the resulting elementary embedding, then $j_{\alpha}(j_E(\kappa))=\alpha$. Letting $k:= j_{\alpha}\circ j_E$, one has that $k:V\to M_{\alpha}$ is a $\lambda$-strong elementary embedding with critical point $\kappa$ and with $k(\kappa)\in C^{(n)}$. We have thus proved the following.

\begin{proposition}
\label{prop2}
Every $\lambda$-strong cardinal  is $\lambda$-$C^{(n)}$-strong, for all $n$. Hence, every strong cardinal is $C^{(n)}$-strong, for every $n$.
\end{proposition}

So for measurable or strong cardinals $\kappa$, the requirement that $j(\kappa)$ belongs to $C^{(n)}$ for the corresponding elementary embeddings $j:V\to M$ does not yield stronger large cardinal notions. But as we shall see next the situation changes completely in the case of superstrong embeddings, that is, when $j$ is such that $V_{j(\kappa)}\subseteq M$.
In the following sections we will analyze the notion of $C^{(n)}$-cardinal at various levels of the usual large cardinal hierarchy, beginning with superstrong cardinals and up to rank-into-rank embeddings. In almost all cases we will show that the corresponding $C^{(n)}$-cardinals form a finer hierarchy. The notion of $C^{(n)}$-cardinal will prove especially useful in the region between supercompact cardinals and Vope\v{n}ka's Principle (VP). There we will establish new equivalences between the existence of $C^{(n)}$-extendible cardinals, restricted forms of VP, and natural reflection principles for classes of structures.  Some important regions of the $C^{(n)}$-cardinal hierarchy have not yet been explored, e.g., $C^{(n)}$-Woodin cardinals; some need further study, e.g., $C^{(n)}$-supercompact cardinals; and there are still many open questions, e.g., whether the $C^{(n)}$-supercompact cardinals form a hierarchy, or the exact relationship between $C^{(n)}$-supercompact and $C^{(n)}$-extendible cardinals. Further work along these lines is already under way.

Let us point out that the consistency of the existence of all the $C^{(n)}$-cardinals to be considered in this paper, except for those in the last section, follows from the consistency of the existence of an $E_0$ cardinal, that is, a cardinal $\kappa$ for which there exists a non-trivial elementary embedding $j:V_\delta \to V_\delta$, some $\delta$, with $\kappa =crit(j)$. In  $V_\delta$, $\kappa$ is     $C^{(n)}$-superstrong,   $C^{(n)}$-extendible, $C^{(n)}$-supercompact, $C^{(n)}$-$k$-huge, and $C^{(n)}$-superhuge, for all $n,k\geq 1$  (Theorem \ref{lasttheorem}, Section \ref{lastsection}). 

\section{$C^{(n)}$-superstrong cardinals}

We shall see next that in the case of superstrong cardinals $\kappa$, the requirement that $j(\kappa)\in C^{(n)}$, for $n>1$,  produces a hierarchy of ever stronger large cardinal principles.

\begin{definition}
A cardinal $\kappa$ is \emph{$C^{(n)}$-superstrong} if there exists an elementary embedding $j:V\to M$, $M$ transitive, with critical point $\kappa$, $V_{j(\kappa)}\subseteq M$, and $j(\kappa)\in C^{(n)}$.
\end{definition}

Notice that if $j:V\to M$ witnesses that $\kappa$ is $C^{(n)}$-superstrong, then $V_\kappa$ is an elementary substructure of $V_{j(\kappa)}$, and therefore $\kappa \in C^{(n)}$. Thus, every $C^{(n)}$-superstrong cardinal belongs to $C^{(n)}$.

\begin{proposition}
\label{prop4}
If $\kappa =crit (j)$, where $j:V\to M$ is an elementary embedding, with $M$ transitive and  $V_{j(\kappa )}\subseteq M$,  then $j(\kappa )\in C^{(1)}$. Hence, every superstrong cardinal is $C^{(1)}$-superstrong.
\end{proposition}

\begin{proof}
Since  $\kappa \in C^{(1)}$, $M$ satisfies that $j(\kappa)\in C^{(1)}$, i.e., $M$ satisfies that $j(\kappa)$ is a strong limit cardinal and $V_{j(\kappa)}=H_{j(\kappa)}$. But  since $(V_{j(\kappa)})^M=V_{j(\kappa)}$,  $j(\kappa)$ is, in $V$,  a strong limit cardinal with $V_{j(\kappa)}=H_{j(\kappa)}$, so  $j(\kappa )\in C^{(1)}$.
\end{proof}

Observe that for $n\geq 1$, the sentence ``$\kappa$ is $C^{(n)}$-superstrong" is $\Sigma_{n+1}$, for   $\kappa$ is $C^{(n)}$-superstrong if and only if  
$$\exists \beta \exists \mu\exists E  (\kappa < \beta <\mu \wedge \mu \in C^{(n)}\wedge E\mbox{ is a }(\kappa ,\beta)\mbox{-extender}\wedge  E\in V_{\mu}\,  \wedge$$ $$ V_{\mu}\models ``j_E(\kappa)\in C^{(n)} \wedge V_{j_E(\kappa)}\subseteq M_E").$$


\begin{proposition}
\label{prop3}
For every $n\geq 1$, if $\kappa$ is $C^{(n+1)}$-superstrong, then there is a $\kappa$-complete normal ultrafilter $\mathcal{U}$ over $\kappa$ such that $$\{ \alpha <\kappa :\alpha \mbox{ is $C^{(n)}$-superstrong}\}\in \mathcal{U}.$$ 
Hence,  the first $C^{(n)}$-superstrong cardinal $\kappa$, if it exists, is not $C^{(n+1)}$-superstrong.
\end{proposition}

\begin{proof}
Suppose $\kappa$ is $C^{(n+1)}$-superstrong, 
witnessed by a $(\kappa , \beta)$-exten\-der $E$ with associated elementary embedding $j_E=j:V\to M$ such that $\beta =j(\kappa)$ and $V_{j(\kappa)}\subseteq M$. Since $j(\kappa)\in C^{(n+1)}$, 
$$V_{j(\kappa)}\models ``\kappa\mbox{ is $C^{(n)}$-superstrong}".$$
And since $\kappa \in C^{(n+1)}$, $M\models ``j(\kappa)\in C^{(n+1)}"$.
Hence, since $V_{j(\kappa)}=(V_{j(\kappa)})^M$, and since ``$\kappa$ is $C^{(n)}$-superstrong" is a $\Sigma_{n+1}$ statement, we have:
$$M\models ``\kappa \mbox{ is $C^{(n)}$-superstrong}".$$
Now using a standard argument (see, e.g., \cite{K}, 5.14, 5.15, or 22.1) one can show that the set 
$\{ \alpha <\kappa: \alpha \mbox{ is $C^{(n)}$-superstrong}\}$ belongs to the $\kappa$-complete normal ultrafilter $\mathcal{U}:= \{ X\subseteq \kappa: \kappa \in j(X)\}$.
\end{proof}

The following Proposition gives an upper bound on the relative position of $C^{(n)}$-superstrong cardinals in the usual large cardinal hierarchy. Recall that $\kappa$ is \emph{$\lambda$-supercompact} if there is an elementary embedding $j:V\to M$, with $M$ transitive, $crit(j)=\kappa$, $j(\kappa)>\lambda$, and $M$ closed under $\lambda$-sequences. Equivalently, $\kappa$ is $\lambda$-supercompact if there exists a $\kappa$-complete, fine, and normal ultrafilter over $\mathcal{P}_\kappa (\lambda)$ (see \cite{K}, 22.7). $\kappa$ is \emph{supercompact} if it is $\lambda$-supercompact for  all $\lambda$.

\begin{proposition}
\label{supercompactimpliessuperstrong}
If $\kappa$ is $2^{\kappa}$-supercompact and belongs to $C^{(n)}$, then there is a $\kappa$-complete normal ultrafilter $\mathcal{U}$ over $\kappa$ such that the set of $C^{(n)}$-superstrong cardinals smaller than $\kappa$ belongs to $\mathcal{U}$.
\end{proposition}

\begin{proof}
Let $j:V\to M$ be an elementary embedding coming from a $\kappa$-complete fine and normal ultrafilter $\mathcal{V}$ on $\mathcal{P}_\kappa (2^\kappa)$.
Let $j^\ast :=j\restriction V_{\kappa +1}$. So, $j^\ast:V_{\kappa +1}\to M_{j(\kappa)+1}$ is elementary and  $j^\ast \in M$. Hence  $M\models ``j^\ast :V_{\kappa +1}\to V_{j(\kappa)+1}$ is elementary". Since $\kappa\in C^{(n)}$, also $M\models ``j(\kappa)\in C^{(n)}"$. Thus, $M\models ``\kappa \mbox{ is }\kappa +1$-$C^{(n)}\mbox{-extendible}"$ (see Definition \ref{definition1} below). Hence, 
if $\mathcal{U}$ is the standard ultrafilter over $\kappa$ derived from $j$, we have
$$\{ \alpha <\kappa : \alpha \mbox{ is } \alpha +1\mbox{-}C^{(n)}\mbox{-extendible}\}\in \,  \mathcal{U}.$$
Now as in \cite{K}, Proposition 26.11 (a), one can show that if $\alpha$ is $\alpha +1$-$C^{(n)}$-extendible, then $\alpha$ is $C^{(n)}$-superstrong.
\end{proof}

\section{$C^{(n)}$-extendible cardinals}
\label{sectionextendible}
Recall that a cardinal $\kappa$  is  \emph{$\lambda$--extendible} if there is an elementary embedding $j:V_{\lambda}\to V_{\mu}$, some $\mu$, with critical point $\kappa$ and such that $j(\kappa)>\lambda$.
And $\kappa$ is  \emph{extendible} if it is $\lambda$--extendible for all $\lambda >\kappa$.

The next lemma implies that every extendible cardinal is supercompact.

\begin{lemma}[M. Magidor \cite{M}]
\label{Magidor}
Suppose $j:V_\lambda \to V_\mu$ is elementary, $\lambda$ is a limit ordinal, and $\kappa$ is the critical point of $j$. Then $\kappa$ is $<\lambda$-supercompact.
\end{lemma}


\begin{proof}
Fix $\gamma <\lambda$ and define
$$\mathcal{U}_\gamma =\{ X\subseteq \mathcal{P}_\kappa (\gamma): j''\gamma \in j(X)\}.$$
Note that this makes sense if $j(\kappa )>\gamma$, in which case it is easy to check that $\mathcal{U}_\gamma$ is a $\kappa$-complete, fine, and  normal measure. Otherwise, let $j^1 =j$ and $j^{m+1}=j\circ j^m$. If $j^m(\kappa )>\gamma$  for some $m$, then define $\mathcal{U}_\gamma$ using $j^m$ instead of $j$.  But such an $m$ does exist, for otherwise $\delta := sup_m(j^m(\kappa ))\leq \gamma <\lambda$, and then since $j(\delta)=\delta$ we would have $j\restriction V_{\delta +2}:V_{\delta +2}\to V_{\delta +2}$ is elementary with critical point $\kappa$, contradicting Kunen's Theorem (\cite{Ku}; see also \cite{K}, 23.14).
\end{proof}


\begin{definition}
\label{definition1}
For a cardinal $\kappa$  and $\lambda >\kappa$, we say that $\kappa$ is \emph{$\lambda$-$C^{(n)}$-extendible} if there is an elementary embedding $j:V_{\lambda}\to V_{\mu}$, some $\mu$, with critical point $\kappa$, and such that $j(\kappa)>\lambda$ and $j(\kappa)\in C^{(n)}$.

We say that $\kappa$ is  \emph{$C^{(n)}$-extendible} if it is $\lambda$-$C^{(n)}$-extendible for all $\lambda >\kappa$.
\end{definition}


\begin{proposition}
\label{proposition3}
Every extendible cardinal is $C^{(1)}$-extendible.
\end{proposition}

\begin{proof}
Suppose $\kappa$ is extendible and $\lambda$ is greater than $\kappa$. Pick $\lambda'\geq \lambda$ in  $C^{(1)}$, and let $j:V_{\lambda'}\to V_{\mu}$ be an elementary embedding with $crit(j)=\kappa$ and $j(\kappa)>\lambda'$. Since $\lambda'$  is a cardinal and $V_{\lambda'}=H_{\lambda'}$, by elementarity of $j$ we  also have that $\mu$ is a cardinal and $V_{\mu}=H_\mu$. Hence $\mu \in C^{(1)}$. And since, again by elementarity, $V_{\mu}\models j(\kappa)\in C^{(1)}$, it follows that $j(\kappa)\in C^{(1)}$. 
\end{proof}

Notice that if $j:V_\lambda \to V_\mu$ has critical point $\kappa$, and $\kappa ,\lambda,  \mu \in C^{(n)}$, then $j(\kappa)\in C^{(n)}$  follows automatically.

Clearly,  if $\kappa$ is $C^{(n)}$-extendible, then  $\kappa \in C^{(n)}$. But more is true.

\begin{proposition}
\label{3.5}
If $\kappa$ is $C^{(n)}$-extendible, then $\kappa\in C^{(n+2)}$.
\end{proposition}

\begin{proof}
By induction on $n$. Every extendible cardinal is in $C^{(3)}$ (see \cite{K}, 23.10), which takes care of the cases $n=0$ and $n=1$. Now suppose $\kappa$ is $C^{(n)}$-extendible and $\exists x \varphi (x)$ is a $\Sigma_{n+2}$ sentence, where $\varphi$ is $\Pi_{n+1}$ and has parameters in $V_{\kappa}$. If $\exists x\varphi(x)$ holds in $V_{\kappa}$, then since by the induction hypothesis $\kappa \in C^{(n+1)}$, we have that $\exists x\varphi(x)$ holds in $V$.
Now suppose $a$ is such that $\varphi(a)$ holds in $V$. Pick $\lambda >\kappa$ with $a\in V_{\lambda}$, and let $j:V_{\lambda}\to V_{\mu}$ be elementary, with critical point $\kappa$ and with $j(\kappa)>\lambda$. Then since $j(\kappa)\in C^{(n)}$, and since $\varphi(a)$ is a $\Pi_{n+1}$ sentence in the parameter $a\in V_{j(\kappa)}$, we have that $V_{j(\kappa)}\models \varphi (a)$, and therefore by elementarity, $V_{\kappa}\models \exists x\varphi (x)$.
\end{proof}



Let us observe that for any given $\alpha <\lambda$, the relation ``$\alpha$ is $\lambda$-$C^{(n)}$-extendible" is $\Sigma_{n+1}$ (for $n\geq 1$), for it holds if and only if   
$$\exists \mu \exists j(j:V_{\lambda}\to V_{\mu}\wedge j\mbox{ elementary}\; \wedge crit(j)=\alpha \wedge j(\alpha)>\lambda\wedge j(\alpha) \in C^{(n)}).$$
Hence,  $``x$ is a $C^{(n)}$-extendible cardinal" is a $\Pi_{n+2}$ property of $x$.  

\begin{proposition}
For every $n\geq 1$, if $\kappa$ is $C^{(n)}$-extendible and $\kappa +1$-$C^{(n+1)}$-extendible,  then the set of $C^{(n)}$-extendible cardinals is unbounded below $\kappa$. 
Hence,  the first $C^{(n)}$-extendible cardinal $\kappa$, if it exists, is not $\kappa +1$-$C^{(n+1)}$-extendible. In particular, the first extendible cardinal $\kappa$ is not $\kappa +1$-$C^{(2)}$-extendible.
\end{proposition}

\begin{proof}
Suppose $\kappa$ is $C^{(n)}$-extendible and  $\kappa +1$-$C^{(n+1)}$-extendible, 
witnessed by $j:V_{\kappa +1}\to V_{j(\kappa)+1}$.  Since $j(\kappa)\in C^{(n+1)}$, 
$$V_{j(\kappa)}\models ``\kappa\mbox{ is $C^{(n)}$-extendible}".$$
Hence, for every $\alpha <\kappa$,
$$V_{j(\kappa)}\models ``\exists \beta >\alpha(\beta \mbox{ is $C^{(n)}$-extendible})",$$
since this is witnessed by $\kappa$. By the elementarity of $j$, for every fixed $\alpha <\kappa$, there is $\beta >\alpha$ such that, 
$$V_{\kappa}\models ``\beta >\alpha\wedge \beta \mbox{ is $C^{(n)}$-extendible}".$$
And since, by Proposition \ref{3.5},  $\kappa \in C^{(n+2)}$, $\beta$ is $C^{(n)}$-extendible in $V$.
\end{proof}

\begin{proposition}
For every $n$, if there exists a $C^{(n+2)}$-extendible cardinal, then there exists a proper  class of $C^{(n)}$-extendible cardinals. 
\end{proposition}

\begin{proof}
By the last proposition, if $\kappa$ is $C^{(n+2)}$-extendible, then the set of $C^{(n)}$-extendible cardinals is unbounded below $\kappa$. Now the proposition follows easily from the fact that if $\kappa$ is $C^{(n+2)}$-extendible, then $\kappa \in C^{(n+4)}$ (Proposition \ref{3.5}), and the fact that being $C^{(n)}$-extendible is a $\Pi_{n+2}$-property.
\end{proof}


Note however that the existence of a $C^{(n+1)}$-extendible cardinal $\kappa$  does not imply the existence of a $C^{(n)}$-extendible cardinal greater than $\kappa$.  For if $\lambda$ is the least such $C^{(n)}$-cardinal, then $V_\lambda$ is a model of ZFC plus ``$\kappa$ is $C^{(n+1)}$-extendible",  because $\lambda \in C^{(n+2)}$ (Proposition \ref{3.5}) and being $C^{(n+1)}$-extendible  is a $\Pi_{n+3}$ property of $\kappa$. And $V_\lambda$ also satisfies that ``there is no $C^{(n)}$-extendible cardinal greater than $\kappa$", because any such $C^{(n)}$-extendible cardinal would be $C^{(n)}$-extendible in $V$, since $\kappa\in C^{(n+2)}$. 

The next proposition gives an upper bound on $C^{(n)}$-superstrong cardinals.

\begin{proposition}
\label{extendibleimpliessuperstrong}
If $\kappa$ is $\kappa +1$-$C^{(n)}$-extendible, then $\kappa$ is $C^{(n)}$-superstrong, and there is a $\kappa$-complete normal ultrafilter $\mathcal{U}$ over $\kappa$ such that the set of $C^{(n)}$-superstrong cardinals smaller than $\kappa$ belongs to $\mathcal{U}$.
\end{proposition}

\begin{proof}
As in \cite{K}, Proposition 26.11 (a).
\end{proof}

\section{Vop\v{e}nka's Principle}


This section builds on results from \cite{BCMR}, giving new and sharper characterizations of Vop\v{e}nka's Principle in terms of $C^{(n)}$-extendible cardinals.

Recall that \emph{Vop\v{e}nka's Principle (VP)}   states that for every proper class $\Ce$ of structures of the same type, there exist $A\ne B$ in $\Ce$ such that $A$ is elementarily embeddable into $B$.

VP can be formulated in the first-order language of set theory as an axiom schema, i.e., as an infinite set of axioms, one for each formula with two free variables. Formally, for each such formula $\varphi (x,y)$ one has the axiom:
$$\forall x[( \forall y\forall z(\varphi (x,y)\wedge \varphi(x,z)\to y\mbox{ and }z\mbox{ are structures of the same type})\wedge$$ $$ \forall \alpha \in OR\; \exists y (rank(y)>\alpha \wedge \varphi(x,y))\to$$
$$\exists y\exists z(\varphi(x,y)\wedge \varphi(x,z)\wedge y\ne z \wedge \exists e(e:y\to z \mbox{ is elementary}))].$$
Henceforth, VP will be understood as this axiom schema.

The theory ZFC plus VP  implies, for instance, that the class of extendible cardinals is stationary, i.e., every definable club proper class contains an extendible cardinal (\cite{M}). And its consistency is known to follow from the consistency of ZFC plus the existence of an almost-huge cardinal (see \cite{K}, or \cite{J2}). We will give below the exact equivalence in terms of $C^{(n)}$-cardinals.

\medskip

Let us consider the following variants of VP, the first one apparently much stronger than the second.

We say that a class $\Ce$ is $\mathbf{\Sigma_n}$ ($\mathbf{\Pi_n}$) if it is definable, with parameters, by a $\Sigma_n$ ($\Pi_n$) formula of the language of set theory. If no parameters are involved, then we use the  lightface types $\Sigma_n$ ($\Pi_n$).

\begin{definition}
\label{variantsVP}
If $\Gamma$ is one of $\mathbf{\Sigma_n}$, $\mathbf{\Pi_n}$,  some  $n\in \omega$, and $\kappa$ is an infinite cardinal, then we write  $VP(\kappa, \Gamma )$ for the following assertion: 

\emph{For every $\Gamma$ proper class $\Ce$ of structures of the same type $\tau$ such that both $\tau$ and the parameters of some $\Gamma$-definition of $\Ce$, if any, belong to $H_\kappa$, $\Ce$ \emph{reflects below $\kappa$}, i.e., 
for every $B\in \Ce$, there exists $A\in \Ce \cap H_\kappa$ that is elementarily embeddable into $B$.}

If $\Gamma$ is one of $\mathbf{\Sigma_n}$, $\mathbf{\Pi_n}$, or $\Sigma_n$, $\Pi_n$, some  $n\in \omega$, we write $VP(\Gamma)$ for the following statement: 

\emph{For every $\Gamma$ proper class $\Ce$ of structures of the language of set theory with one (equivalently, finitely-many) additional $1$-ary relation symbol(s),  there exist distinct $A$ and $B$ in $\Ce$ with an elementary embedding of $A$ into $B$.}
\end{definition}

VP for $\mathbf{\Sigma_1}$ classes is a consequence of ZFC. In fact, the following holds.

\begin{theorem}
\label{vpsigma1}
If $\kappa$ is an uncountable cardinal, then every (not necessarily proper) class $\Ce$ of structures of the same type $\tau \in H_\kappa$ which is $\Sigma_1$ definable, with parameters in $H_\kappa$, reflects below $\kappa$. Hence, $VP(\kappa ,\mathbf{\Sigma_1})$ holds for every uncountable cardinal $\kappa$.
\end{theorem}

\begin{proof}
Fix an uncountable cardinal $\kappa$ and a class $\mathcal{C}$ of structures of the same type $\tau\in H_\kappa$, definable by a $\Sigma_1$ formula with parameters in $H_\kappa$. 

Given $B\in \mathcal{C}$, let $\lambda$ be a regular cardinal greater than $\kappa$, with $B\in H_\lambda$, and let $N$ be an elementary substructure of $H_\lambda$, of cardinality less than $\kappa$, which contains $B$ and the transitive closure of $\{ \tau \}$ together with the parameters involved in some $\Sigma_1$ definition of $\mathcal{C}$.

Let $A$ and $M$ be the transitive collapses of $B$ and $N$, respectively, and let $j:M\to N$ be the collapsing isomorphism. Then $A\in H_\kappa$,  and $j\restriction A:A\to B$ is an elementary embedding. Observe that $j(\tau)=\tau$. So,  since $\Sigma_1$ formulas are upwards absolute for transitive models, and since $M\models A\in \mathcal{C}$, we have that $A\in\mathcal{C}$.
\end{proof}


In contrast, Vop\v{e}nka's Principle for $\Pi_1$ proper classes implies the existence of very large cardinals.

\begin{theorem}
\label{vppi1}
$ $
\begin{enumerate}
\item 
If $VP(\Pi_1)$ holds, then there exists a supercompact cardinal.
\item 
If $VP(\mathbf{\Pi_1})$ holds, then there is a proper class of supercompact cardinals.
\end{enumerate}
\end{theorem}

\begin{proof}
(1). Let $\mathcal{C}$ be the class of structures of the form $\langle V_{\lambda +2 }, \in ,   \alpha , \lambda  \rangle$, where 
$\lambda$ is the least limit ordinal greater than $\alpha$ such that no $\kappa \leq \alpha$ is $<\lambda$-supercompact.

We claim that $\mathcal{C}$ is $\Pi_1$ definable without parameters. For $X\in \mathcal{C}$ if and only if $X=\langle X_0,X_1,X_2,X_3 \rangle$, where

 \begin{enumerate}
 \item[(1)] $X_2$ is an ordinal
 \item[(2)] $X_3$ is a limit ordinal greater than $X_2$
 \item[(3)] $X_0=V_{X_3 +2}$
 \item[(4)] $X_1 =\in \restriction X_0$
 \item[(5)] And the following hold in $\langle X_0, X_1 \rangle$:
 \begin{enumerate}
 \item $\forall \kappa \leq X_2 (\kappa \mbox{ is not } <X_3\mbox{-supercompact})$
 \item $\forall \mu (\mu \mbox{ limit}\wedge X_2 <\mu <X_3 \to \exists \kappa \leq X_2 (\kappa \mbox{ is} <\mu\mbox{-supercompact}))$.
 \end{enumerate}
   \end{enumerate}
   
 
If there is no supercompact cardinal, then $\mathcal{C}$ is a proper class. So by $VP(\Pi_1)$, there exist  structures $\langle V_{\lambda+2 }, \in,   \alpha ,\lambda  \rangle \ne \langle V_{\mu+2 },\in,   \beta ,\mu  \rangle$ in $\Ce$ and an elementary embedding $$j:\langle V_{\lambda+2 }, \in,   \alpha,\lambda \rangle \to \langle V_{\mu+2 },\in,   \beta ,\mu  \rangle.$$ 
Since  $j$ must send $\alpha$ to $\beta$ and $\lambda$ to $\mu$, $j$ is not the identity, for otherwise the two structures would be equal. Hence by Kunen's Theorem  (\cite{Ku}; see also \cite{K}, 23.14) we must have $\lambda <\mu$. By the way $\lambda$ and $\mu$ are uniquely defined from $\alpha$ and $\beta$, respectively, this implies that while no $\kappa \leq \alpha$ is $<\lambda$-supercompact, there is some $\kappa \leq \beta$ which is $<\lambda$-supercompact, and therefore  $\alpha <\beta$. So, $j$ has  critical point some $\kappa \leq\alpha$.  It now follows by Lemma \ref{Magidor} that $\kappa$ is $<\lambda$-supercompact. But this is impossible because $\langle V_{\lambda +2 }, \in ,   \alpha , \lambda  \rangle \in \mathcal{C}.$

(2). 
Fixing an ordinal $\xi$, to show that there is a supercompact cardinal greater than $\xi$, we argue as above. The only difficulty now is to ensure that $\kappa >\xi$. But this can be achieved by  letting $\mathcal{C}$ be the class of structures of the form $\langle V_{\lambda +2 }, \in ,   \alpha , \lambda, \{ \gamma \}_{\gamma \leq \xi}  \rangle$, where $\alpha >\xi$ and 
$\lambda$ is the least limit ordinal greater than $\alpha$ such that no $\kappa \leq \alpha$ is $<\lambda$-supercompact. The class $\mathcal{C}$ is now $\Pi_1$ definable with $\xi$ as an additional parameter. If there is no supercompact cardinal above $\xi$, then $\Ce$ is a proper class. So arguing as before we have an elementary embedding $j$ between two different structures in $\Ce$, which now must be the identity on the ordinals less than or equal to $\xi$, so that $j$ has  critical point some $\kappa$ with $\xi <\kappa$. A contradiction then follows as before.
\end{proof}

For any given $\Pi_1$ class of structures $\mathcal{C}$ of the same type one may wonder how much supercompactness is needed to guarantee that VP holds for $\Ce$. An upper bound is given in the next Proposition.

Let us say that a limit ordinal $\lambda$ \emph{captures} a proper class $\mathcal{C}$ if the class of ordinal ranks of elements of $\mathcal{C}$, intersected with $\lambda$, is unbounded in $\lambda$. I.e., 
for every $\alpha$ less than $\lambda$ there exists $A\in \Ce$ of rank strictly between $\alpha$ and  $\lambda$.

Note that if $\mathcal{C}$ is $\mathbf{\Pi_n}$, then every $\lambda$ in $C^{(n+1)}$ greater than the rank of the parameters involved in a $\Pi_n$ definition of $\mathcal{C}$ captures $\mathcal{C}$. For if $\alpha <\lambda$, then the assertion that there is a structure in $\Ce$ of rank greater than $\alpha$ can be written as a $\Sigma_{n+1}$ sentence with parameter $\alpha$ and the parameters of some definition of $\Ce$. And since this sentence is true in $V$, and $\lambda \in C^{(n+1)}$, it is also true in $V_\lambda$.
Notice also that, by a similar argument, every cardinal in $C^{(2)}$ belongs to the $\Pi_1$ definable class $Lim(C^{(1)})$ of all limit points of $C^{(1)}$ and, moreover, it captures all $\Pi_1$ proper classes. 
However,  the least ordinal $\lambda$ in $Lim(C^{(1)})$ that captures all $\Pi_1$ proper classes is strictly less than the least ordinal $\mu$ in $C^{(2)}$.
The point is that, fixing an enumeration $\langle \varphi_n(x):n<\omega\rangle$ of all $\Pi_1$ formulas that define proper classes,  the sentence
$$\exists \lambda \exists x(\lambda\in Lim(C^{(1)})\wedge x=V_{�\lambda}\, \wedge$$
$$ \forall n (V_{\lambda}\models \forall \alpha \exists \beta >\alpha \exists a(rk(a)>\beta \wedge \models_1 \varphi_n(a))))$$
is $\Sigma_2$ in the parameter $\langle \varphi_n(x):n<\omega\rangle$, and so it is reflected by $\mu$, thereby producing a  $\lambda < \mu$ in $Lim(C^{(1)})$ that captures all $\Pi_1$ proper classes.

\begin{proposition}
\label{scimpliesvp}
Let $\mathcal{C}$ be a $\Pi_1$ proper class of structures of the same type. If there exists a cardinal $\kappa$ that is $<\lambda$-supercompact, for some $\lambda\in Lim(C^{(1)})$   greater than $\kappa$ that captures $\mathcal{C}$, then  VP holds for $\mathcal{C}$.
\end{proposition}

\begin{proof}
Since $\lambda$ captures $\mathcal{C}$,  in $V_{\lambda}$ there exist elements of $\mathcal{C}$ of arbitrarily high rank.
So, since $\lambda \in Lim(C^{(1)})$, we can find $\delta <\lambda$  such that $V_{\delta}=H_\delta$, and $B\in \mathcal{C}\cap V_{\delta}$ of rank greater than $\kappa$. Let  $j:V\to M$ be an elementary embedding with critical point $\kappa$, with $j(\kappa)>\delta$, and $M$ closed under $\delta$-sequences.
Since $B\in M$ and $\mathcal{C}$ is $\Pi_1$ definable, $M\models ``B\in \mathcal{C}"$. And since $M$ is closed under $\delta$-sequences, the elementary embedding $j\restriction B:B\to j(B)$ belongs to $M$. Thus,
$$M\models ``\exists A\in \mathcal{C}\, \exists e(rank(A)<j(\kappa)\wedge e:A\to j(B)\mbox{ is elementary}),$$
since this is witnessed by $B$ and $j\restriction B$.

By elementarity, the same must hold in $V$, namely,
$$\exists A\in \mathcal{C}\, \exists e(rank(A)<\kappa\wedge e:A\to B\mbox{ is elementary}),$$
which is what we wanted.
\end{proof}

We give next a strong converse to Theorem \ref{vppi1}.
\begin{theorem}[\cite{BCMR}]
\label{theorem1}
Suppose that $\Ce$ is a $\mathbf{\Sigma_2}$ (not necessarily proper) class of structures of the same type $\tau$, and suppose 
that there exists a supercompact cardinal $\kappa$ larger than the rank of the parameters that 
appear in some  $\Sigma_2$ definition of~$\Ce$, and with $\tau \in V_\kappa$. Then for every $B\in \Ce$ 
there exists $A\in \Ce \cap V_{\kappa}$ that is elementarily embeddable into~$B$.
\end{theorem}

\begin{proof}
Fix a $\Sigma_2$ formula $\varphi (x,y)$ and a set $b$ such that $\Ce=\{ B:\varphi (B,b)\}$, 
and suppose that $\kappa$ is a supercompact cardinal with $b\in V_{\kappa}$. 
Fix $B\in \Ce$, and let $\lambda  \in C^{(2)}$ be greater than $\text{rank}(B)$. Let $j\colon V\to M$ be an elementary embedding with $M$ transitive and  
critical point~$\kappa$, such that $j(\kappa)>\lambda$ and $M$ is closed under $\lambda$-sequences. Thus, $B$ and $j\restriction B:B\to j(B)$ are in~$M$, and also 
$V_{\lambda}\in M$.  Hence $V_{\lambda}\preceq_{1}M$. Moreover, since $j(\tau)=\tau$, $j(B)$ is a structure of type $\tau$, and $j\restriction B$ is an elementary embedding.

Since $V_{\lambda}\preceq_{2}V$, $V_{\lambda}\models \varphi (B,b)$. And since 
$\Sigma_2$ formulas are upwards absolute between $V_{\lambda}$ and~$M$, $M\models \varphi (B,b)$.

Thus, in $M$ it is true that there exists $X\in M_{j(\kappa)}$ such that $\varphi(X,b)$, namely $B$, 
and there exists an elementary embedding $e\colon X\to j(B)$, namely $j\restriction B$. 
Therefore, by elementarity, the same holds in~$V$; that is, 
there exists $X\in V_{\kappa}$ such that $\varphi (X,b)$, and there exists an elementary embedding 
$e\colon X\to B$.
\end{proof}

The following corollaries  give characterizations of Vop\v{e}nka's principle for $\Pi_1$ and $\Sigma_2$ classes in terms of supercompactness. The equivalence of (2) and (3) in the next two corollaries was already proved in \cite{BCMR}.

\begin{corollary}
\label{theorem1-1}
The following are equivalent:
\begin{enumerate}
\item $VP(\Pi_1)$.
\item $VP(\kappa, \mathbf{\Sigma}_2)$, for some $\kappa$.
\item There exists a supercompact cardinal.
\end{enumerate}
\end{corollary}

\begin{proof}
(2)$\Rightarrow$(1) is immediate. (1)$\Rightarrow$(3) is given by Theorem \ref{vppi1}, (1). And (3)$\Rightarrow$(2) follows from Theorem \ref{theorem1}.
\end{proof}

The next corollary gives the parameterized version. The implication (1)$\Rightarrow$(3) is given by Theorem \ref{vppi1}, (2). 

\begin{corollary}
\label{theorem1-2}
The following are equivalent:
\begin{enumerate}
\item $VP(\mathbf{\Pi}_1)$.
\item $VP(\kappa, \mathbf{\Sigma}_2)$, for a proper class of cardinals $\kappa$.
\item There exists a proper class of supercompact cardinals.
\end{enumerate}
\end{corollary}

We shall give next a characterization of supercompactness in terms of a natural principle of reflection. 
 Recall from Definition \ref{variantsVP} that a cardinal $\kappa$ \emph{reflects} a class of structures $\Ce$ of the same type if for every $B\in \Ce$ there exists $A\in \Ce \cap H_\kappa$ which is elementarily embeddable into $B$.


\begin{theorem}[Magidor \cite{M}]
\label{theoremMagidor}
If $\kappa$ is the least cardinal that reflects the $\Pi_1$ proper class $\Ce$  of structures of the form $\langle V_{\lambda},\in \rangle$, then $\kappa$ is supercompact. 
\end{theorem}

\begin{proof}
Let $\alpha <\kappa$ be such that there is  an elementary embedding $j : V_{\alpha +1}\to V_{\lambda +1}$ for some  singular $\lambda\in C^{(3)}$. Since $j$ sends $\alpha$ to $\lambda$, it must have a critical point $\beta$, which must be smaller than $\alpha$, for if $\beta =\alpha$, then $\alpha$ would be regular in $V_{\alpha +1}$ (being the critical point of the embedding), and so by elementarity $\lambda$ would be regular in $V_{\lambda +1}$, hence regular in $V$, contrary to our choice of $\lambda$.

By Lemma \ref{Magidor}, $\beta$ is $<\alpha$-supercompact, and so $V_\alpha \models ``\beta$ is supercompact". By elementarity of $j$, $V_\lambda \models ``j (\beta)$ is supercompact", hence, since $\lambda \in C^{(3)}$, and being supercompact is $\Pi_3$ expressible,  $j (\beta)$ is supercompact.

Thus  $j(\beta) \geq \kappa$, because $j(\beta)$ reflects $\Ce$, by Theorem \ref{theorem1}, and $\kappa$ is the least cardinal that does this. Now suppose, aiming for a contradiction, that  $j(\beta) >\kappa$. Then 
$$V_{j(\beta)}\models ``\kappa \mbox{ reflects the class }\mathcal{C}".$$
Hence, by elementarity of   $j$,
$$V_\beta \models ``\gamma \mbox{ reflects the class }\mathcal{C}"$$
for some $\gamma <\beta$. And since $\gamma$ is fixed by $j$, 
$$V_{j(\beta)}\models ``\gamma \mbox{ reflects the class }\mathcal{C}".$$ 
Making use of the fact that $V_{j(\beta)} \preceq_{\Sigma_2}V$, it follows that $\gamma$ reflects the class $\mathcal{C}$, thus contradicting the minimality of $\kappa$.
\end{proof}

The last two theorems yield the following characterizations of the first supercompact cardinal.

\begin{corollary}
The following are equivalent:
\begin{enumerate}
\item $\kappa$ is the first supercompact cardinal.
\item $\kappa$ is the least cardinal that reflects all $\Sigma_2$ definable, with parameters in $V_\kappa$, classes of structures of the same type. i.e., $\kappa$ is the least ordinal for which $VP(\kappa , \mathbf{\Sigma_2})$ holds.
\item $\kappa$ is the least cardinal that reflects the $\Pi_1$ class of structures of the form $\langle V_\lambda ,\in \rangle$, $\lambda$ an ordinal.
\end{enumerate}
\end{corollary}

\begin{proof}
If $\kappa$ is a supercompact cardinal, then by Theorem \ref{theorem1} $VP(\kappa ,\mathbf{\Sigma_2})$ holds, and therefore $\kappa$ reflects the class of structures $\langle V_\lambda ,\in \rangle$, $\lambda$ an ordinal. So by Theorem \ref{theoremMagidor}, (1), (2), and (3) are equivalent.
\end{proof}

The following parameterized version of the last Corollary has been pointed out by David Asper\'o. 
\begin{corollary}
A cardinal $\kappa$ reflects all $\mathbf{\Pi}_1$ (proper) classes of structures of the same type if and only if either $\kappa$ is a supercompact cardinal or a limit of supercompact cardinals.
\end{corollary}

\begin{proof}
Clearly, the property of reflecting $\mathbf{\Pi}_1$ classes of structures is closed under limits. So if $\kappa$ is a supercompact cardinal or a limit of supercompact cardinals, then Theorem \ref{theorem1} implies that  $\kappa$ reflects all $\mathbf{\Pi_1}$ classes. The other direction can be proved as in  Theorem  \ref{vppi1} (2).
\end{proof}

\medskip


We will prove next similar results for classes of higher complexity, for which we shall need $C^{(n)}$-extendible cardinals.

\begin{theorem}
\label{theorem4}
For every $n\geq 1$, if $\kappa$ is a $C^{(n)}$--extendible cardinal, then every class $\Ce$ of structures of the same type $\tau \in H_\kappa$ which is $\Sigma_{n+2}$ definable, with parameters in $H_\kappa$, reflects below $\kappa$. Hence $VP(\kappa, \mathbf{\Sigma}_{n+2})$ holds.
\end{theorem}

\begin{proof}
Fix a $\Sigma_{n+2}$ formula $\exists x \varphi (x,y,z)$, where $\varphi$ is $\Pi_{n+1}$, such that for some set $b\in V_{\kappa}=H_\kappa$, 
$$\Ce :=\{ B:\exists x \varphi (x, B,b)\}$$ is a  class of structures of the same type $\tau\in H_\kappa$.   

Fix $B\in \Ce$ and let $\lambda \in C^{(n+2)}$ be greater than $\kappa$ and the rank of $B$. Thus, 
$$V_{\lambda}\models \exists x \varphi (x,B,b).$$

Let $j:V_{\lambda}\to V_{\mu}$  be an elementary embedding with critical point $\kappa$, with $j(\kappa)>\lambda$, and $j(\kappa)\in C^{(n)}$. Note that $B$ and $j\restriction B:B\to j(B)$ are in $V_{\mu}$. Moreover, since $j$ fixes $\tau$,  $j(B)$ is a structure of type $\tau$, and $j\restriction B$ is an elementary embedding.

As $\kappa,\lambda\in C^{(n+2)}$ (see Proposition \ref{3.5}), it follows  that $V_{\kappa}\preceq_{n+2}V_{\lambda}$. So we have
$$V_{\lambda}\models ``\forall x\in V_{\kappa} \forall \theta \in \Sigma_{n+2} (V_{\kappa}\models \theta (x) \leftrightarrow \models_{n+2} \theta (x))".$$
Hence, by elementarity,
$$V_{\mu}\models ``\forall x\in V_{j(\kappa)} \forall \theta \in \Sigma_{n+2} (V_{j(\kappa)}\models \theta (x) \leftrightarrow \models_{n+2} \theta (x))",$$
which says that  $V_{j(\kappa)}\preceq_{n+2} V_{\mu}$.

Since $j(\kappa)\in C^{(n)}$, we also have $V_{\lambda}\preceq_{n+1} V_{j(\kappa)}$, and therefore $V_{\lambda}\preceq_{n+1} V_{\mu}$. It follows that $V_{\mu}\models \exists x \varphi (x,B,b)$, because $V_\lambda \models \exists x \varphi (x,B,b)$. 

Thus, in $V_{\mu}$ it is true that there exists $X\in V_{j(\kappa)}$ such that $X\in \Ce$, namely $B$, and there exists an elementary embedding $e:X\to j(B)$, namely $j\restriction B$.  Therefore, by the elementarity of $j$, the same is true in $V_{\lambda}$, that is,  there exists $X\in V_{\kappa}$ such that $X\in \Ce$, and there exists an elementary embedding $e:X\to B$. Let $A\in V_{\kappa}$ be such an $X$, and let $e:A\to B$ be an elementary embedding. Since $\lambda\in C^{(n+2)}$, $A\in \Ce$, and we are done.
\end{proof}
%

The next theorem  yields a strong converse to Theorem \ref{theorem4}. 


\medskip

The notion of $C^{(n)}$-extendibility used in \cite{BCMR} has the following apparently stronger form -- let us call it $C^{(n)+}$-extendibility: For $\lambda \in C^{(n)}$, a cardinal
$\kappa$ is \emph{$\lambda$-$C^{(n)+}$-extendible} if it is $\lambda$-$C^{(n)}$-extendible, witnessed by some $j:V_{\lambda}\to V_{\mu}$ which, in addition to satisfying $j(\kappa)>\lambda$ and $j(\kappa)\in C^{(n)}$, also satisfies that $\mu \in C^{(n)}$. 
$\kappa$ is \emph{$C^{(n)+}$-extendible} if it is $\lambda$-$C^{(n)+}$-extendible for every $\lambda >\kappa$ with $\lambda \in C^{(n)}$.

Every extendible cardinal is $C^{(1)+}$-extendible (see  the proof of Proposition \ref{proposition3}). We shall see below that the first $C^{(n)}$-extendible cardinal is $C^{(n)+}$-extendible, for all $n$.

\begin{theorem}
\label{vppin+1}
Suppose $n\geq 1$. If $VP(\Pi_{n+1})$ holds, then there exists a $C^{(n)+}$-extendible cardinal.
\end{theorem}

\begin{proof}
Suppose there are no $C^{(n)+}$-extendible cardinals. Then the class function $F$   on the ordinals given by:
\begin{quote}
$F(\alpha)=$ the least $\lambda\in C^{(n+1)}$ greater than $\alpha$ such that $\alpha$ is not $\lambda$-$C^{(n)+}$-extendible,
\end{quote}
is defined for all ordinals $\alpha$.

Let $C=\{ \eta >0:  \forall \alpha <\eta \; F(\alpha )<\eta \}$. So $C$ is a closed unbounded proper class of ordinals, and is  contained in $C^{(n+1)}$ because every $\eta\in C$ is the supremum of the set $\{ F(\alpha):\alpha <	\eta\} \subseteq C^{(n+1)}$.

We claim that $C$ is $\Pi_{n+1}$ definable, without parameters. It is sufficient to see that $F$ is $\Pi_{n+1}$ definable. We have: $\lambda=F(\alpha)$ if and only if 
\begin{enumerate}
\item $\lambda\in C^{(n+1)}$
\item $\alpha <\lambda$
\item $\forall \beta >\lambda (\beta\in C^{(n)} \to V_{\beta}\models (\alpha \mbox{ is not } \lambda\mbox{-$C^{(n)+}$-extendible}))$, and 
\item $V_{\lambda}\models \forall \lambda' >\alpha (\lambda' \in C^{(n+1)} \to (\alpha \mbox{ is } \lambda'\mbox{-$C^{(n)+}$-extendible}))$.
\end{enumerate}
 The point is that, for any $\alpha <\lambda'$, the relation ``$\alpha$ is $\lambda'$-$C^{(n)+}$-extendible" is $\Sigma_{n+1}$,   for it holds if and only if    
$$\exists \mu \exists j(j:V_{\lambda'}\to V_{\mu}\mbox{ is elementary}\; \wedge crit(j)=\alpha \wedge j(\alpha)>\lambda'\wedge j(\alpha), \mu  \in C^{(n)}).$$
 So it holds in $V$ if and only if it holds in $V_{\lambda}$, for any $\lambda \in C^{(n+1)}$ greater than $\lambda'$. And if it holds in $V_{\beta}$, with $\beta\in C^{(n)}$, then it holds in $V$. Moreover, since $\lambda\in C^{(n+1)}$, for every $\lambda'<\lambda$ we have $\lambda'\in C^{(n+1)}$ if and only if $V_{\lambda}\models \lambda'\in C^{(n+1)}$.

Since the conjunction of statements (1)-(4) above is $\Pi_{n+1}$, it follows that $F$, and therefore also $C$, is $\Pi_{n+1}$ definable. Let $\varphi$ be a $\Pi_{n+1}$ formula that defines $C$.

For each ordinal $\alpha$, let $\lambda_{\alpha}$ be the least limit point of $C$ greater than $\alpha$. We have that $x=\lambda_{\alpha}$ if and only if  $x$ is an ordinal greater than $\alpha$ that belongs to $C$ and is such that
\begin{enumerate}
\item $V_x\models \forall \beta \exists \gamma (\gamma >\beta \wedge \varphi(\gamma))$
\item $V_x\models \forall \beta (\beta >\alpha\to \exists \gamma<\beta \forall \eta (\gamma <\eta<\beta \to \neg \varphi(\eta))),$
\end{enumerate}
which shows that the function $\alpha \longmapsto \lambda_{\alpha}$ is $\Pi_{n+1}$ definable.

Consider now the proper class $\mathcal{C}$ of structures $\mathcal{A}_{\alpha}$ of the form $$\langle V_{\lambda_{\alpha} }, \in ,    \alpha ,\lambda_{\alpha},  C\cap \alpha +1  \rangle,$$
where $\alpha \in C$. 

We claim that $\mathcal{C}$ is $\Pi_{n+1}$ definable. 
We have: $X\in \mathcal{C}$ if and only if $X=\langle X_0, X_1, X_2 ,X_3,X_4 \rangle$, where

 \begin{enumerate}
 \item $X_2 \in C$
 \item $X_3 =\lambda_{X_2}$
 \item $X_0=V_{X_3}$
 \item $X_1=\in \restriction X_0$
 \item $X_4=C\cap X_2 +1$
   \end{enumerate}
We have already seen that (1) and (2) are $\Pi_{n+1}$ expressible. And so are (3) and (4), as one can easily see. As for (5), note that $X_4=C\cap X_2 +1$ holds in $V$ if and only if 
it holds in $V_{X_3}$. So (5) is equivalent to
$$V_{X_3}\models \forall x(x\in X_4 \leftrightarrow \varphi (x)\wedge x\in X_2 +1)$$
which is $\Pi_{n+1}$ expressible.
 
So by $VP(\Pi_{n+1})$ there exist  $\alpha \ne \beta$  and an elementary embedding $$j:\mathcal{A}_{\alpha}\to \mathcal{A}_{\beta}.$$
Since  $j$ must send $\alpha$ to $\beta$, $j$ is not the identity. So $j$ has  critical point some $\kappa \leq\alpha$. 

We claim that $\kappa \in  C$. Otherwise, $\gamma :=sup(C\cap \kappa )<\kappa$. Let $\delta$ be the least ordinal in $C$ greater than $\gamma$ such that $\delta <\lambda_{\alpha}$. So $\kappa <\delta\leq \alpha$.
Since $\delta$ is definable from $\gamma$ in $\mathcal{A}_{\alpha}$,  and since $j(\gamma)=\gamma$, we must also have $j(\delta)=\delta$. 
But then $j\restriction V_{\delta+2}:V_{\delta+2}\to V_{\delta+2}$ is a nontrivial elementary embedding, contradicting Kunen's Theorem.

Since $\lambda_\alpha \in C^{(n+1)}$, $V_{\lambda_\alpha}\models \varphi(\kappa)$. Hence by elementarity,  $V_{\lambda_\beta}\models \varphi (j(\kappa ))$. So since $\lambda_\beta \in C^{(n+1)}$, it follows that $j(\kappa)\in C$.

Note that since $\lambda_\alpha \in C$, we have $\kappa < F(\kappa )<\lambda_\alpha$. Thus,
$$j\restriction V_{F(\kappa)}:V_{F(\kappa)}\to V_{j(F(\kappa))}$$ is elementary, with critical point $\kappa$. 

And since $j(\kappa)\in C$, $F(\kappa)<j(\kappa)$.
Moreover, by the elementarity of $j$, $V_{\lambda_{\beta}}$ satisfies that $j(F(\kappa))$ belongs to $C^{(n+1)}$, and so since $\lambda_{\beta}\in C^{(n+1)}$ this is true in $V$. 
This shows that $j\restriction V_{F(\kappa)}$ witnesses that $\kappa$ is $F(\kappa)$-$C^{(n)+}$-extendible. But this is impossible by the definition of $F$.
\end{proof}

The proof of the last theorem can be easily adapted to prove  the parameterized version:  if $VP(\mathbf{\Pi}_{n+1})$  holds, then there is a proper class of $C^{(n)}$-extendible cardinals.  Fixing an ordinal $\xi$, to show that there is a $C^{(n)}$-extendible cardinal greater than $\xi$, we argue as above. To ensure that $\kappa >\xi$, we now let $\mathcal{C}$ be the class of structures of the form $$\langle V_{\lambda_\alpha}, \in ,   \alpha , \lambda_\alpha, C\cap \alpha +1 , \{ \gamma \}_{\gamma \leq \xi}  \rangle$$ where $\alpha >\xi$ and $\alpha \in C$. 
The class $\mathcal{C}$ is now $\Pi_{n+1}$ definable with $\xi$ as an additional parameter. If there is no $C^{(n)}$-extendible cardinal above $\xi$, then $\Ce$ is a proper class. So arguing as before we have an elementary embedding $j$ between two different structures in $\Ce$, which now must be the identity on the ordinals less than or equal to $\xi$, and so $j$ has  critical point some $\kappa$ with $\xi <\kappa$. A contradiction then follows as before.

 The following corollaries summarize the results above. The equivalences of (2) and (4) in the next two corollaries were already proved in \cite{BCMR}.

\begin{corollary}
\label{coro4.13}
The following are equivalent for $n\geq 1$:
\begin{enumerate}
\item $VP( \Pi_{n+1})$.
\item $VP(\kappa, \mathbf{\Sigma}_{n+2})$, for some $\kappa$.
\item There exists a  $C^{(n)}$-extendible cardinal.
\item There exists a $C^{(n)+}$-extendible cardinal.
\end{enumerate}
\end{corollary}

\begin{proof}
(1)$\Rightarrow$(4) follows from Theorem \ref{vppin+1}. (4)$\Rightarrow$(3) and (2)$\Rightarrow$(1) are immediate. And (3)$\Rightarrow$(2) follows from Theorem \ref{theorem4}.
\end{proof}

In particular, since every extendible cardinal is $C^{(1)}$-extendible (Proposition \ref{proposition3}), we have the following.

\begin{corollary}
$ $
The following are equivalent:
\begin{enumerate}
\item $VP( \Pi_{2})$.
\item $VP(\kappa, \mathbf{\Sigma}_{3})$, for some $\kappa$.
\item There exists an extendible cardinal.
\item There exists a $C^{(1)+}$-extendible cardinal.
\end{enumerate}
\end{corollary}

We finally obtain the following characterization of VP. The equivalence of (2), (3), and (5) was already proved in \cite{BCMR}.

\begin{corollary}
\label{coro3.15}
The following are equivalent:
\begin{enumerate}
\item $VP(\Pi_n)$, for every $n$.
\item $VP(\kappa, \mathbf{\Sigma}_n)$,  for a proper class of cardinals $\kappa$, and for every $n$.
\item VP.
\item For every $n$, there exists a $C^{(n)}$-extendible cardinal.
\item For every $n$, there exists a $C^{(n)+}$-extendible cardinal.
\end{enumerate}
\end{corollary}

\begin{proof}
Clearly, (3) implies (1) and (2) implies (3). All the other implications follow immediately from Corollary \ref{coro4.13}.
\end{proof}

Since the consistency of VP follows from the consistency of the existence of an almost-huge cardinal (see \cite{K}, 24.18 and Section \ref{sectionhuge} below), an almost-huge cardinal gives an upper bound in the usual large cardinal hierarchy on the consistency strength of $C^{(n)}$-extendible cardinals, all $n\geq 1$.

\medskip

We give next a  characterization of $C^{(n)}$-extendible cardinals in terms of reflection of classes of structures. 

\begin{theorem}
\label{reflectionextendibles}
If $n\geq 1$ and  $\kappa$ is the least cardinal that reflects all $\Pi_{n+1}$ proper classes of structures of the same type, then $\kappa$ is $C^{(n)+}$-extendible.
\end{theorem}

\begin{proof}
Suppose otherwise. Then by \ref{theorem4} there is no $C^{(n)}$-extendible, and therefore no $C^{(n)+}$-extendible, cardinal less than or equal to $\kappa$. For such a cardinal would reflect all $\Sigma_{n+2}$ (hence all $\Pi_{n+1}$) classes of structures, contradicting the minimality of $\kappa$.

Consider the class $\Ce$ of structures of the form $\langle V_\xi , \in , \lambda ,\alpha , C^{(n)}\cap \xi \rangle$, where $\alpha <\lambda <\xi$, and 
\begin{enumerate}
\item $\lambda \in C^{(n)}$ 
\item $\xi \in Lim(C^{(n)})$
\item the cofinality of $\xi$ is uncountable
\item
$\forall \beta <\xi \, \forall \mu (\exists j (j:V_\lambda \to V_\mu \wedge crit (j)=\alpha \wedge j(\alpha )=\beta )\to$
$\exists j'\exists \mu'(j':V_\lambda \to V_{\mu'} \wedge \mu'<\xi \wedge V_\xi \models ``\mu'\in C^{(n)}" \wedge crit(j')=\alpha \wedge j'(\alpha)=\beta)),$ and
\item $\lambda$ witnesses that no ordinal less than or equal to $\alpha$ is $\lambda$-$C^{(n)+}$-extendible.
\end{enumerate}
Clearly, $\Ce$ is a $\Pi_{n+1}$ definable proper class. So  there exists an elementary embedding 
$$j:\langle V_{\xi'},\in ,\lambda',\alpha', C^{(n)}\cap \xi'\rangle \to \langle V_\xi ,\in ,\lambda ,\kappa ,C^{(n)}\cap \xi \rangle$$
with both $\langle V_{\xi'},\in ,\lambda',\alpha' , C^{(n)}\cap \xi' \rangle$ and $\langle V_\xi ,\in ,\lambda ,\kappa ,  C^{(n)}\cap \xi \rangle$ in $\Ce$,  and with $\langle V_{\xi'},\in ,\lambda',\alpha' , C^{(n)}\cap \xi' \rangle$ of rank less than $\kappa$. So $\xi'<\kappa$.

Let $\alpha =crit(j)$. We claim that $\alpha \in C^{(n)}$. Otherwise, let $\gamma :=sup (C^{(n)}\cap \alpha )$. So $\gamma <\alpha$. Let $\delta \in C^{(n)}$ be the least such that $\gamma <\delta <\xi'$. Since $\delta$ is definable from $\gamma$ in $\langle V_{\xi'},\in ,\lambda',\alpha', C^{(n)}\cap \xi'\rangle$ and $j(\gamma )=\gamma$, also $j(\delta )=\delta$. Hence $j\restriction V_{\delta +2}: V_{\delta +2}\to V_{\delta +2}$ is elementary, contradicting Kunen's Theorem. 

 If $j^m(\alpha )<\xi'$ for all $m$, then $\{ j^m(\alpha )\}_{m\in \omega}\in V_{\xi'}$, because $\xi'$ has uncountable cofinality, contradicting Kunen's Theorem. So  for some $m$ we have $j^m(\alpha)<\xi'\leq j^{m+1}(\alpha)$.

We claim that there exists an elementary embedding $k:V_{\lambda'}\to V_\mu$, some $\mu\in C^{(n)}$, with $crit(k)=\alpha$ and $k(\alpha)=j^{m+1}(\alpha)$. We prove this by induction on $i\leq m$. For $i=0$ take $k=j\restriction V_{\lambda'}$. Now suppose it true for $i<m$. Since $j^{i+1}(\alpha)<\xi'$, by (4) above there exist $j'$ and $\mu'$ such that $j':V_{\lambda'}\to V_{\mu'}$ is elementary, $\mu'<\xi'$, $V_{\xi'}\models \mu'\in C^{(n)}$, $crit(j')=\alpha$, and $j'(\alpha )=j^{i+1}(\alpha)$. Notice that since $\xi\in C^{(n)}$ and, by the elementarity of $j$,  $V_\xi \models j(\mu')\in C^{(n)}$, it follows that  $j(\mu')\in C^{(n)}$. 
Composing $j'$ with $j$ we now have  that $k:= (j\circ j'):V_{\lambda'}\to V_{j(\mu')}$ is elementary,  has critical point $\alpha$, and $k(\alpha)=j^{i+2}(\alpha)$, as desired. 

Note that since $\alpha ,\xi',\xi\in C^{(n)}$, it follows that $j^{m+1}(\alpha)\in C^{(n)}$. Thus, $k$ witnesses that $\alpha$ is $\lambda'$-$C^{(n)+}$-extendible, contradicting (5) above.
\end{proof}

Observe that if $\kappa$ is the least $C^{(n)}$-extendible cardinal, then by Theorem \ref{theorem4} it reflects all $\mathbf{\Sigma_{n+2}}$ classes of structures. Hence, by the theorem above, $\kappa$ is the least cardinal that does this, and therefore $\kappa$ is $C^{(n)+}$-extendible. 

\begin{corollary}
The following are equivalent for each $n\geq 1$:
\begin{enumerate}
\item $\kappa$ is the least $C^{(n)}$-extendible cardinal.
\item $\kappa$ is the least cardinal that reflects all $\Sigma_{n+2}$ definable, with parameters in $V_\kappa$, classes of structures of the same type. I.e., $\kappa$ is the least ordinal for which $VP(\kappa , \mathbf{\Sigma_{n+2}})$ holds.
\item $\kappa$ is the least cardinal that reflects all $\Pi_{n+1}$  proper classes of  structures of type $\langle V_\alpha ,\in ,A\rangle$, where $A$ is a unary predicate.
\end{enumerate}
\end{corollary}

\begin{proof}
If $\kappa$ is a $C^{(n)}$-extendible cardinal, then by Theorem \ref{theorem4} $VP(\kappa ,\mathbf{\Sigma_{n+2}})$ holds, and therefore $\kappa$ reflects all $\Pi_{n+1}$  proper classes of  structures of type $\langle V_\alpha ,\in ,A\rangle$, where $A$ is a unary predicate. Now, the proof of Theorem \ref{reflectionextendibles} shows that if $\kappa$ is the least cardinal that reflects a particular  $\Pi_{n+1}$ definable proper class $\Ce$ of structures of the form $\langle V_\xi , \in , \lambda ,\alpha , C^{(n)}\cap \xi \rangle$, where $\alpha <\lambda <\xi$, then $\kappa$ is $C^{(n)+}$-extendible, and therefore $C^{(n)}$-extendible. Thus, since the triple $\langle \lambda ,\alpha , C^{(n)}\cap \xi\rangle$ can be easily coded as a subset of $V_\xi$, the equivalence of (1), (2), and (3) follows immediately from Theorem \ref{reflectionextendibles}.
\end{proof}

The following parameterized version also follows.

\begin{theorem}
\label{reflectionextendiblesbold}
A cardinal $\kappa$ reflects all $\mathbf{\Pi}_{n+1}$ (proper) classes of structures of the same type if and only if either $\kappa$ is a $C^{(n)}$-extendible cardinal or a limit of $C^{(n)}$-extendible cardinals.
\end{theorem}

\begin{proof}
It is clear that the property of reflecting $\mathbf{\Pi_{n+1}}$ classes of structures is closed under limits. So if $\kappa$ is a $C^{(n)}$-extendible cardinal or a limit of $C^{(n)}$-extendible cardinals, then Theorem \ref{theorem4} implies that  $\kappa$ reflects all $\mathbf{\Pi_{n+1}}$ classes of structures. The other direction can be proved similarly as in  Theorem  \ref{reflectionextendibles}. For suppose $\kappa$ reflects all $\mathbf{\Pi}_{n+1}$ (proper) classes of structures and is neither $C^{(n)}$-extendible nor a limit of $C^{(n)}$-extendible cardinals. Then for some ordinal $\eta <\kappa$ there is no $C^{(n)}$-extendible cardinal greater than $\eta$ and less than or equal to $\kappa$. 

Consider the class $\Ce$ of structures of the form $\langle V_\xi , \in , \lambda ,\alpha , C^{(n)}\cap \xi , \{ \eta'\}_{\eta' \leq \eta}\rangle$, where $\eta <\alpha <\lambda <\xi$, which satisfy (1)-(4) from the proof of Theorem \ref{reflectionextendibles} and also 
\begin{enumerate}
\item[(5)] $\lambda$ witnesses that no ordinal less than or equal to $\alpha$ and greater than $\eta$ is $\lambda$-$C^{(n)}$-extendible.
\end{enumerate}
Clearly, $\Ce$ is a $\Pi_{n+1}$ definable proper class, with $\eta$ as a parameter. So  there exists an elementary embedding 
$$j:\langle V_{\xi'},\in ,\lambda',\alpha', C^{(n)}\cap \xi', \{ \eta'\}_{\eta' \leq \eta} \rangle \to \langle V_\xi ,\in ,\lambda ,\kappa ,C^{(n)}\cap \xi, \{ \eta'\}_{\eta' \leq \eta}\rangle$$
with both structures in $\Ce$,  and with $\langle V_{\xi'},\in ,\lambda',\alpha' , C^{(n)}\cap \xi', \{ \eta'\}_{\eta' \leq \eta} \rangle$ of rank less than $\kappa$. So $\xi'<\kappa$ and $\eta <crit(j)$. The rest of the proof proceeds now as in Thereom \ref{reflectionextendibles}.
\end{proof}

We finish this section with the following observation.  Suppose $n\geq 1$. Given a $\Sigma_{n+1}$ definable class of structures $\mathcal{C}$, say via the $\Sigma_{n+1}$ formula $\varphi (x)$, let $\mathcal{C}^\ast$ be the class of structures of the form $A^\ast =\langle V_\alpha , \in ,\alpha ,A\rangle$, where $\alpha$ is the least ordinal in $C^{(n)}$ such that $V_\alpha \models \varphi (A)$. If $A\in \Ce$, then such an $\alpha$ exists, because the set of ordinals $\alpha$ such that $V_\alpha \models \varphi (A)$ is club. Conversely, if $\langle V_\alpha ,\in, \alpha, A\rangle \in \Ce^\ast$, then $V_\alpha \models \varphi (A)$ and $\alpha \in C^{(n)}$, which implies that $\varphi (A)$ holds in $V$, and so $A\in \Ce$. Thus, we have that 
$$A\in \mathcal{C}\mbox{  if and only if  }A^\ast \in \mathcal{C}^\ast.$$ 
Now notice that $\mathcal{C}^\ast$ is $\Pi_n$ definable. This explains why, e.g., $VP(\Pi_n)$ is equivalent to $VP(\Sigma_{n+1})$, or why a cardinal reflects $\Pi_n$ classes if and only if it reflects $\Sigma_{n+1}$ classes.

\section{$C^{(n)}$-supercompact cardinals}

Let us consider next the $C^{(n)}$-cardinal form of supercompactness.

\begin{definition}
\label{definitionCnsupercompact}
If  $\kappa$  is a cardinal and $\lambda >\kappa$, we say that $\kappa$ is \emph{$\lambda$-$C^{(n)}$-supercompact} if there is an elementary embedding $j:V\to M$, with $M$ transitive, such that $crit(j)=\kappa$, $j(\kappa)>\lambda$, $M$ is closed under $\lambda$-sequences,  and $j(\kappa)\in C^{(n)}$.

We say that $\kappa$ is  \emph{$C^{(n)}$-supercompact} if it is $\lambda$-$C^{(n)}$-supercompact for every  $\lambda >\kappa$.

\end{definition}

If $\kappa$ is $C^{(n)}$-superhuge (see Section \ref{sectionhuge}), then $\kappa$ is $C^{(n)}$-supercompact. Thus, it follows  from Proposition \ref{BDPT} below  that if $\kappa$ is $C^{(n)}$-$2$-huge, then there is a $\kappa$-complete normal ultrafilter $\mathcal{U}$ over $\kappa$ such that $\{ \alpha <\kappa : V_\kappa \models ``\alpha $ is $C^{(n)}$-supercompact$"\}\in \mathcal{U}$.

\medskip

The notion of $\lambda$-$C^{(n)}$-supercompact\-ness, unlike $\lambda$-supercompactness, cannot be formulated in terms of normal measures on $\mathcal{P}_\kappa(\lambda)$. The problem is that if $j:V\to M$ is an ultrapower embedding coming from such a measure, then $2^{\lambda^{<\kappa}} <j(\kappa) < (2^{\lambda^{<\kappa}})^+$ (see \cite{K}, 22.11), and so $j(\kappa)$ is not a cardinal. So, in order to formulate the notion of $\lambda$-$C^{(n)}$-supercompactness in the first-order language of set theory we will make use of \emph{long} extenders\footnote{I want to thank Ralf Schindler for illuminating discussions on long extenders.} having as support a sufficiently rich transitive set  (see \cite{MS} for a presentation of \emph{short} extenders of this kind). 

So  suppose $j:V\to M$ witnesses that $\kappa$ is $\lambda$-$C^{(n)}$-supercompact. Let $Y$ be a transitive subset of $M$ which contains $j\restriction \lambda$, is closed under sequences of length $\leq \lambda$,  and is closed under $j$. Let $\zeta$ be the least ordinal such that $Y\subseteq j(V_\zeta )$. For each $a\in [Y]^{<\omega}:=\{ x \subseteq Y: x\mbox{ is finite}\}$, let $E_a$ be  defined by:
$$X\in E_a \mbox{ if and only if }X\subseteq (V_{\zeta})^a \mbox{ and }j^{-1}\restriction j(a)\in j(X).$$
Note that the function $j^{-1}\restriction j(a) : j(a)\to a$  is an $\in$-isomorphism  that sends $j(x)$ to $x$, for every $x\in a$.

It is not difficult to check that  the sequence $E:=\langle E_a:a\in [Y]^{<\omega}\rangle$ is an \emph{extender over $V_\zeta$ with critical point $\kappa$ and  support $Y$}, that is,
\begin{enumerate}
\item Each $E_a$ is a $\kappa$-complete ultrafilter over $(V_{\zeta})^a$, and $E_{\{ \kappa \}}$ is not $\kappa^+$-complete.
\item If $a\subseteq b$ and $X\in E_a$, then $\{ f\in (V_{\zeta})^b:f\restriction a\in X\}\in E_b$.
\item For every $a$, the set $\{f:a\to range(f) : f$ is an $\in$-isomorphism$\}$ belongs to $E_a$.
\item If $F:(V_{\zeta})^a\to V$ is such that $\{ f: F(f)\in \bigcup (range(f))\}\in E_a$, then there is $z\in Y$ such that $\{ f\in (V_{\zeta})^{a\cup \{ z\}}: F(f\restriction a)=f(z)\}\in E_{a\cup \{ z\}}$.
\item The ultrapower $Ult(V, E)$ is well-founded.
\end{enumerate}
To check (5), observe that the map $k:Ult(V,E)\to M$ given by $k([a, [f]])=j(f)(j^{-1}\restriction j(a))$ is an elementary embedding. 

If $j_E:V\to M_E\cong Ult(V,E)$ is the corresponding ultrapower embedding, then $j=k\circ j_E$. Moreover,  $Y\subseteq M_E$   and $k$ is the identity on $Y$  (see \cite{MS} for details). Since $Y$ was assumed to be closed under $j$, it easily follows that $Y$ is also closed under $j_E$.  Hence, $j\restriction Y=j_E\restriction Y$. For if $y\in Y$, then $j(y)=k(j_E(y))=j_E(y)$. In particular, $\kappa = crit(j_E)$ and $j(\kappa)=j_E(\kappa)$. 

We claim that $j_E$ witnesses the $\lambda$-$C^{(n)}$-supercompactness of $\kappa$, for which it only remains  to check that $M_E$ is closed under $\lambda$-sequences.
First note that 
$$M_E=\{ j_E(f)(j^{-1}\restriction j(a)): a\in [Y]^{<\omega} \mbox{ and } f:(V_{\zeta})^a\to V\}.$$
For if $x=[a,[f]]\in M_E$, then writing $s$ for $j^{-1}\restriction j(a)$ and noticing that  $k(s)=s$, because $s\in Y$, we have
$$k(x)=j(f)(s)=k\circ j_E(f)(s)=k(j_E(f))(k(s))=k(j_E(f)(s)) $$
and since $k$ is one-to-one we have that $x=j_E(f)(s)$.

Now fix $j_E(f_i)(j^{-1}\restriction j(a_i))$, for $i<\lambda$. Let $f=\langle f_i:i<\lambda\rangle$,  let $c$ be  $j_E\restriction \lambda$, and let $d=\langle j^{-1}\restriction j(a_i) :i<\lambda \rangle$ . Notice that since $Y$ is closed under $j$ and under $\lambda$ sequences, $d\in Y$. Set $b =\{ c,d\}$ and let $F:(V_{\zeta})^{b}\to V$ be defined as follows:
if $s=\{ s_c,s_d\} \in (V_{\zeta})^{b}$ is such that $s_c$ and $s_d$ are functions with the same ordinal $\alpha$ as their domain, then $F(s)$ is the function $g$ with domain $\alpha$ such that $g(i)= f(s_c(i))(s_d(i))$, whenever $s_c(i)$ is an ordinal and $s_d(i)\in (V_{\zeta})^{a_{s_c(i)}}$, and $g(i)=0$  otherwise. Otherwise, $F(s)=0$. 
Then, noticing that $j^{-1}\restriction j(b)$ maps $j(c)$ to $j_E\restriction \lambda$ and $j(d)$ to $d$, we have:
$$j_E(F)(j^{-1}\restriction j(b))(i)=j_E(f)(j_E(i))(j^{-1}\restriction j(a_i))=j_E(f_i)(j^{-1}\restriction j(a_i))$$
and so $\langle j_E(f_i)(j^{-1}\restriction j(a_i)): i<\lambda   \rangle =j_E(F)(j^{-1}\restriction j(b)) \in M_E$.

\medskip

Conversely, suppose $Y$ is transitive and $E=\langle E_a:a\in [Y]^{<\omega}\rangle$ is an extender over some $V_\zeta$ with critical point $\kappa$ and support $Y$. If $j_E:V\to M_E \cong Ult(V,E)$ is the corresponding ultrapower embedding, then $crit(j_E)=\kappa$ and $Y \subseteq M_E$ (see \cite{MS}, Lemmas 1.4 and 1.5).
Moreover, if $[a,[f]]\in M_E$, then
$$j_E(f)(j_E^{-1}\restriction j_E(a))=[a, [c^a_f]]([a, [Id_{(V_\zeta)^a}]])=[a, [f]]$$
where $c^a_f: (V_\zeta )^a\to V$ is the constant function with value $f$, and $Id_{(V_\zeta)^a}:(V_\zeta )^a\to V$ is the identity function. Hence, 
$$M_E=\{ j_E(f)(j_E^{-1}\restriction j_E(a)): a\in [Y]^{<\omega} \mbox{ and } f:(V_\zeta)^a\to V\}.$$ 
So if $Y$  is closed under $j_E$ and under $\lambda$ sequences, then  one can show, as above, that $M_E$ is closed under $\lambda$ sequences.

%

Notice that for every $\mu \in C^{(1)}$ and every  $\kappa ,\zeta , Y, E\in V_\mu$, we have:  $E$ is an extender over $V_\zeta$ with critical point $\kappa$ and support $Y$ if and only if $V_\mu \models ``E$ is an extender over $V_\zeta$ with critical point $\kappa$ and support $Y$". Moreover, $(j_E(\kappa))^{V_\mu}=j_E(\kappa)$.

Thus, for $n\geq 1$, $\kappa$ is $\lambda$-$C^{(n)}$-supercompact if and only if
$$\exists \mu\exists E \exists Y \exists \zeta ( \mu \in C^{(n)}\wedge \lambda , E, Y \in V_\mu \;\wedge Y\mbox{ is transitive }\wedge [Y]^{\leq \lambda} \subseteq Y  \wedge$$ $$V_\mu \models ``E\mbox{ is an extender over $V_\zeta$ with critical point $\kappa$ and support $Y$} \wedge$$ $$  j_E[Y]\subseteq Y  \wedge j_E(\kappa)>\lambda  \wedge j_E(\kappa)\in C^{(n)}").$$

It follows that, for $n\geq 1$, ``$\kappa$ is $\lambda$-$C^{(n)}$-supercompact" is $\Sigma_{n+1}$ expressible. Hence, ``$\kappa$ is $C^{(n)}$-supercompact" is $\Pi_{n+2}$ expressible.

\medskip

Thus, for $n\geq 1$,  if $\kappa$ is $C^{(n)}$-supercompact and $\alpha$ is any ordinal in  $C^{(n+1)}$ greater than $\kappa$, then $V_\alpha \models ``\kappa$ is $C^{(n)}$-supercompact".
Moreover, since for every $n$,  ``$\exists \kappa (\kappa$ is $C^{(n)}$-supercompact$)$" is $\Sigma_{n+3}$ expressible, the first $C^{(n)}$-supercompact cardinal does not belong to $C^{(n+3)}$. But we don't know if, e.g., the first $C^{(1)}$-supercompact cardinal belongs to $C^{(3)}$. We don't know either if the $C^{(n)}$-supercompact cardinals form a hierarchy in a strong sense, that is, if the first $C^{(n)}$-supercompact cardinal is smaller than the first $C^{(n+1)}$-supercompact cardinal, for all $n$. 





Every extendible cardinal is supercompact, and the first extendible is much greater than the first supercompact  (see \cite{K}), but we don't know if, for $n\geq 1$, every $C^{(n)}$-extendible cardinal is $C^{(n)}$-supercompact, or if the first $C^{(n)}$-extendible cardinal is actually greater than the first $C^{(n)}$-supercompact cardinal. However,  since every $C^{(n)}$-extendible cardinal belongs to $C^{(n+2)}$ (Proposition \ref{3.5}), the first $C^{(n)}$-supercompact cardinal is smaller than the first $C^{(n+1)}$-extendible cardinal, assuming both exist. 

So far, the only upper bound we know, in the usual large cardinal hierarchy, on the consistency strength of the existence of $C^{(n)}$-supercompact cardinals, for $n\geq 1$, is the existence of an $E_0$ cardinal (see Section \ref{lastsection}).

\section{$C^{(n)}$-huge and  $C^{(n)}$-superhuge cardinals}
\label{sectionhuge}

Recall that a cardinal $\kappa$ is  \emph{$m$-huge}, for $m\geq 1$,  if it is the critical point of an elementary embedding $j:V\to M$ with $M$ transitive and closed under $j^m(\kappa)$-sequences, where $j^m$ is the $m$-th iterate of $j$. A cardinal is called \emph{huge} if it is $1$-huge.

\begin{definition}
We say that a  cardinal $\kappa$ is \emph{$C^{(n)}$-$m$-huge} ($n\geq 1$) if it is $m$-huge, witnessed by $j$,  with $j(\kappa)\in C^{(n)}$.  We say that $\kappa$ is \emph{$C^{(n)}$-huge} if it is $C^{(n)}$-$1$-huge.
\end{definition}

In contrast with $C^{(n)}$-supercompact cardinals, which do not admit a characterization in terms of ultrafilters, but only in terms of long extenders, $C^{(n)}$-$m$-huge cardinals can be characterized in terms of normal ultrafilters. To wit: $\kappa$ is $C^{(n)}$-$m$-huge if and only if it is uncountable and there is a $\kappa$-complete fine and normal ultrafilter $\mathcal{U}$ over some $\mathcal{P}(\lambda)$ and cardinals $\kappa =\lambda_0 <\lambda_1 <\ldots  <\lambda_m=\lambda$, with $\lambda_1 \in C^{(n)}$, and such that for each $i<m$,
$$\{ x\in \mathcal{P}(\lambda ): ot(x\cap \lambda_{i+1})=\lambda_i\}\in \mathcal{U}.$$
(See \cite{K}, 24.8 for a proof of the case $n=1$, which also works for arbitrary $n$.) It follows that  ``$\kappa$ is $C^{(n)}$-$m$-huge" is $\Sigma_{n+1}$ expressible.

Clearly, every huge cardinal is $C^{(1)}$-huge.  But the first huge cardinal is not $C^{(2)}$-huge. For suppose $\kappa$ is the least huge cardinal and $j:V\to M$ witnesses that $\kappa$ is $C^{(2)}$-huge. Then since ``$x$ is huge" is $\Sigma_{2}$ expressible, we have $$V_{j(\kappa)}\models ``\kappa\mbox{ is huge}".$$ Hence, since $(V_{j(\kappa)})^M=V_{j(\kappa)}$, 
$$M\models ``\exists \delta <j(\kappa)(V_{j(\kappa)}\models ``\delta \mbox{ is huge}")".$$
By elementarity, there is a huge cardinal less than $\kappa$ in $V$, which is absurd.

A similar argument, using that for $n\geq 1$ ``$\kappa$ is $C^{(n)}$-$m$-huge" is $\Sigma_{n+1}$ expressible,  for all $m$, shows that the first $C^{(n)}$-$m$-huge cardinal is not $C^{(n+1)}$-huge, for all $m, n\geq 1$.


\medskip

Analogous considerations can be made in the case of almost-huge cardinals. Recall that a cardinal $\kappa$ is  \emph{almost-huge}  if it is the critical point of an elementary embedding $j:V\to M$ with $M$ transitive and closed under $\gamma$-sequences, for every $\gamma <j(\kappa)$.
Thus, we say that a cardinal $\kappa$ is \emph{$C^{(n)}$-almost-huge} if it is almost-huge, witnessed by an embedding $j$ with $j(\kappa)\in C^{(n)}$.

$C^{(n)}$-almost-huge cardinals can also be  characterized in terms of normal ultrafilters. To wit: $\kappa$ is $C^{(n)}$-almost-huge if and only if there exist an inaccessible $\lambda \in C^{(n)}$ greater than $\kappa$ and a coherent sequence of  normal ultrafilters $\langle \mathcal{U}_\gamma : \kappa \leq \gamma <\lambda\rangle$ over $\mathcal{P}_\kappa (\gamma)$ such that the corresponding embeddings $j_\gamma :V\to M_\gamma \cong Ult (V,\mathcal{U}_\gamma)$ and $k_{\gamma ,\delta}:M_\gamma \to M_\delta$ satisfy: if $\kappa\leq \gamma <\lambda$ and $\gamma \leq \alpha <j_\gamma (\kappa)$, then there exists $\delta$ such that $\gamma \leq \delta <\lambda$ and $k_{\gamma ,\delta} (\alpha )=\delta$. (See \cite{K}, 24.11 for details.) It follows that for $n\geq 1$, ``$\kappa$ is $C^{(n)}$-almost-huge" is $\Sigma_{n+1}$ expressible. Now similar arguments as in the case of $C^{(n)}$-huge cardinals show that for $n\geq 1$, the first $C^{(n)}$-almost-huge cardinal is not $C^{(n+1)}$-almost-huge.
 
Clearly, if $\kappa$ is $C^{(n)}$-huge, then it is $C^{(n)}$-almost-huge. Moreover,  similarly as in Proposition \ref{prop3}, one can show that there is a $\kappa$-complete normal ultrafilter $\mathcal{U}$ over $\kappa$ such that the set of $C^{(n)}$-almost-huge cardinals below $\kappa$ belongs to $\mathcal{U}$. Notice also that every $C^{(n)}$-almost-huge cardinal is $C^{(n)}$-superstrong, and therefore belongs to $C^{(n)}$. Hence, since being $C^{(n)}$-huge is $\Sigma_{n+1}$ expressible, the first $C^{(n)}$-huge cardinal is smaller than the first $C^{(n+1)}$-almost-huge cardinal, provided both exist. 
 
\begin{definition}
We say that a cardinal $\kappa$ is \emph{$C^{(n)}$-superhuge } if and only if for every $\alpha$ there is an elementary embedding $j:V\to M$, with $M$ transitive, such that $crit(j)=\kappa$, $\alpha <j(\kappa)$, $M$ is closed under $j(\kappa)$-sequences, and $j(\kappa)\in C^{(n)}$.
\end{definition}

Clearly, $\kappa$ is superhuge (see \cite{K}) if and only if it is $C^{(1)}$-superhuge. 

\begin{proposition}
If $\kappa$ is $C^{(n)}$-superhuge, then $\kappa \in C^{(n+2)}$.
\end{proposition}

\begin{proof}
Similarly as in Proposition \ref{3.5}.
\end{proof}

Note that $\kappa$ is $C^{(n)}$-superhuge if and only if for every $\alpha$ there is a $\kappa$-complete fine and normal ultrafilter $\mathcal{U}$ over some $\mathcal{P}(\lambda)$, with $\lambda \in C^{(n)}$  greater than $\alpha$ and $\kappa$, so that $\{ x\in \mathcal{P}(\lambda):ot(x)=\kappa \}\in \mathcal{U}$.

Thus, ``$\kappa$ is $C^{(n)}$-superhuge" is $\Pi_{n+2}$ expressible.

Arguing similarly as in the case of $C^{(n)}$-huge cardinals, one can easily see that the first $C^{(n)}$-superhuge cardinal is not $C^{(n+1)}$-superhuge. For suppose $\kappa$ is the least $C^{(n)}$-superhuge cardinal, and suppose, towards a contradiction, that it is $C^{(n+1)}$-superhuge. Let $j:V\to M$ be an elementary embedding with $crit(j)=\kappa$, $V_{j(\kappa)}\subseteq M$, and $j(\kappa)\in C^{(n+1)}$. Then $V_{j(\kappa)}\models ``\kappa$ is $C^{(n)}$-superhuge". Hence, since $(V_{j(\kappa)})^M=V_{j(\kappa)}$, 
$$(V_{j(\kappa)})^M\models ``\exists \delta (\delta \mbox{ is $C^{(n)}$-superhuge})".$$
By elementarity, $$V_\kappa \models ``\exists \delta (\delta \mbox{ is $C^{(n)}$-superhuge})".$$ 
And since $\kappa \in C^{(n+2)}$, it is true in $V$ that there exists a $C^{(n)}$-superhuge cardinal $\delta <\kappa$, contradicting the minimality of $\kappa$.

\medskip

Clearly, every $C^{(n)}$-superhuge cardinal is $C^{(n)}$-supercompact. The following Proposition is the $C^{(n)}$-cardinal version of similar results for $m$-huge and superhuge cardinals due to Barbanel-Di Prisco-Tan \cite{BDPT} (see also \cite{K}, 24.13).

\begin{proposition}
\label{BDPT}
$ $
\begin{enumerate}
\item If $\kappa$ is $C^{(n)}$-superhuge, then it is $C^{(n)}$-extendible. Moreover,  there is a $\kappa$-complete normal ultrafilter $\mathcal{U}$ over $\kappa$ such that $\{ \alpha <\kappa :\alpha $ is $C^{(n)}$-extendible$\}\in \mathcal{U}$.
\item If $\kappa$ is $C^{(n)}$-$2$-huge, then there is a $\kappa$-complete normal ultrafilter $\mathcal{U}$ over $\kappa$ such that $\{ \alpha <\kappa : V_\kappa \models ``\alpha $ is $C^{(n)}$-superhuge$"\}\in \mathcal{U}$.
\end{enumerate}
\end{proposition}

\begin{proof}
(1): Fix $\lambda >\kappa$. Let $j:V\to M$ be a witness to the $C^{(n)}$-hugeness of $\kappa$ with $j(\kappa)>\lambda$. Then $M\models ``\kappa$ is $\lambda$-$C^{(n)}$-extendible". Since $M\models ``j(\kappa)\in C^{(n+2)}"$, we have  $(V_{j(\kappa)})^M\models ``\kappa$ is $\lambda$-$C^{(n)}$-extendible". Hence,  since $(V_{j(\kappa)})^M=V_{j(\kappa)}$,
$V_{j(\kappa)}\models ``\kappa$ is $\lambda$-$C^{(n)}$-extendible", and therefore  $\kappa$ is $\lambda$-$C^{(n)}$-extendible.

Notice that the argument above actually shows that  $(V_{j(\kappa)})^M\models ``\kappa$ is $C^{(n)}$-extendible". Thus, if $\mathcal{U}$ is the standard $\kappa$-complete normal ultrafilter over $\kappa$ derived from $j$, we have $\{ \alpha <\kappa :V_\kappa \models ``\alpha $ is $C^{(n)}$-extendible$"\}\in \mathcal{U}$. Hence, since $\kappa\in C^{(n+2)}$, $\{ \alpha <\kappa :\alpha $ is $C^{(n)}$-extendible$\}\in \mathcal{U}$.

\smallskip

(2): Let $j:V\to M$ witness that $\kappa$ is $C^{(n)}$-$2$-huge. Since $M$ is closed under $j^2(\kappa)$-sequences, the $\kappa$-complete fine and normal ultrafilter over $\mathcal{P}(j(\kappa))$ derived from $j$ that witnesses the hugeness of $\kappa$ belongs to $M$. Hence, 
$$M\models ``\kappa\mbox{ is huge, witnessed by some embedding }k\mbox{ with }k(\kappa)=j(\kappa)".$$ 
Thus, if $\mathcal{U}$ is the standard $\kappa$-complete normal ultrafilter over $\kappa$ derived from $j$, we have $$A:= \{ \alpha <\kappa : \alpha\mbox{ is huge, witnessed by an embedding $k$ with }k(\alpha)=\kappa \} \in \mathcal{U}.$$ Since $M$ contains all the  ultrafilters over $\mathcal{P}(\kappa)$, it follows that for each $\alpha \in A$, $$M\models ``\alpha\mbox{ is huge, witnessed by an embedding }k\mbox{ with }k(\alpha)=\kappa ".$$ Hence, $\{ \beta <\kappa: \alpha$ is huge, witnessed by an embedding $k$ with $k(\alpha)=\beta \}\in \mathcal{U}$. 

Notice that since $\kappa , j(\kappa)\in C^{(n)}$, $V_{j(\kappa)}\models ``\kappa \in C^{(n)}"$. And since $(V_{j(\kappa)})^M=V_{j(\kappa)}$, the set $C^{(n)}\cap \kappa$ is in $\mathcal{U}$. So,
$\{ \beta <\kappa: \alpha$ is $C^{(n)}$-huge, witnessed by an embedding $k$ with $k(\alpha)=\beta \}\in \mathcal{U}$.

Thus, for each $\alpha \in A$, $$V_\kappa \models ``\alpha\mbox{ is }C^{(n)}\mbox{-superhuge}".$$ It follows that $\{ \alpha <\kappa:V_\kappa \models ``\alpha$ is $C^{(n)}$-superhuge$"\} \in \mathcal{U}$.
\end{proof}

\section{On elementary embeddings of a rank into itself}
\label{lastsection}

Finally, we will consider  $C^{(n)}$-cardinal forms of the very strong large cardinal principles known as $E_i$, for $0\leq i \leq \omega$ (see \cite{L}).  The principle $E_0$ (also known in the literature as I3 (see \cite{K}, 24)) asserts the existence of a non-trivial elementary embedding $j:V_{\delta}\to V_{\delta}$, with $\delta$ a limit ordinal. Let us call the critical point of such an embedding an \emph{$E_0$ cardinal}.

If $j:V_{\delta}\to V_{\delta}$ witnesses that $\kappa$ is $E_0$, then Kunen's Theorem implies that  $\delta=sup\{ j^m (\kappa):m\in \omega \}$, where $j^m$ is the m-th iterate of $j$. It follows that $\delta\in C^{(1)}$, because all the $j^m(\kappa)$ are inaccessible cardinals (in fact, measurable cardinals) and therefore they all belong to $C^{(1)}$. Moreover, $V_{\kappa}$ and $V_{j^m(\kappa)}$, all $m \geq 1$, are elementary substructures of $V_{\delta}$. Therefore, $V_{\delta}\models \mbox{ZFC}$. 

\begin{theorem}
\label{lasttheorem}
If $\kappa$ is $E_0$, witnessed by $j:V_{\delta}\to V_{\delta}$, then in $V_\delta$,  $\kappa$ (and also all the cardinals $j^m(\kappa)$, $m\geq 1$), are  $C^{(n)}$-superstrong,   $C^{(n)}$-extendible, $C^{(n)}$-supercompact, $C^{(n)}$-$k$-huge, and $C^{(n)}$-superhuge, for all $n,k\geq 1$. 
\end{theorem}

\begin{proof}
First notice that in $V_\delta$, $\kappa$ and all iterates $j^m(\kappa)$, $m\geq 1$, belong to $C^{(n)}$, for all $n$.  

To see that $\kappa$ is $C^{(n)}$-superhuge in $V_\delta$, pick any $\alpha <\delta$. Then we can find $m$ such that $j^m(\kappa)>\alpha$.  Thus, $j^m:V_\delta \to V_\delta$, $crit(j^m)=\kappa$, $\alpha <j^m(\kappa)$, $V_\delta$ is closed under $j^m(\kappa)$-sequences, and $j^m(\kappa)\in C^{(n)}$.
Define $\mathcal{U}$ by:
$$X\in \mathcal{U}\mbox{ if and only if } X\subseteq \mathcal{P}(j^m(\kappa)) \wedge j^m  ``j^m(\kappa)\in j^m(X).$$
One can easily check that  $\mathcal{U}$ is  a $\kappa$-complete fine and normal ultrafilter $\mathcal{U}$ over  $\mathcal{P}(j^m(\kappa))$ with $\{ x\in \mathcal{P}(j^m(\kappa)):ot(x)=\kappa \}\in \mathcal{U}$. Since $\mathcal{U}\in V_\delta$, we can now, in $V_\delta$, define the ultrapower embedding $k:V_\delta \to M\cong Ult(V_\delta ,\mathcal{U})$. Then $crit(k)=\kappa$, $\alpha <k(\kappa)$, $M$ is closed under $k(\kappa)$-sequences, and $k(\kappa)\in C^{(n)}$ (see \cite{K}, 24.8 for details). Since the same can be done for each $\alpha <\delta$, this shows that in $V_\delta$ $\kappa$ is $C^{(n)}$-superhuge. Hence $\kappa$ is also $C^{(n)}$-supercompact, $C^{(n)}$-extendible, and $C^{(n)}$-superstrong. The argument for showing that $\kappa$ is $C^{(n)}$-$k$-huge is similar to the one for $C^{(n)}$-superhugeness, using the characterization of $C^{(n)}$-$k$-hugeness in terms of ultrafilters (see Section \ref{sectionhuge}).
\end{proof}

Thus, modulo ZFC,  the consistency of  the existence of an $E_0$  cardinal implies the consistency with ZFC of the existence of all $C^{(n)}$-cardinals considered in previous sections. Notice that a consequence of the Theorem above (and Corollary \ref{coro3.15})  is that $V_{\delta}$ satisfies VP.

\medskip

Let us now say that $\kappa$ is a \emph{$C^{(n)}$-$E_0$ cardinal} if it is $E_0$, witnessed by some embedding $j:V_{\delta}\to V_{\delta}$,  with $j(\kappa)\in C^{(n)}$.

Clearly, if $\kappa$ is $C^{(n)}$-$E_0$, then $\kappa\in C^{(n)}$.   In fact,  $\kappa$ is a limit point of $C^{(n)}$. For suppose  $\alpha$ is smaller than $\kappa$. Then $V_{j(\kappa)}$ satisfies that there exists some $\beta \in C^{(n)}$ greater than $\alpha$, since $\kappa$ is such a $\beta$. Hence, by the elementarity of $j$,  $V_\kappa$ satisfies that  some $\beta$ greater than $\alpha$  belongs to $C^{(n)}$. But since $\kappa\in C^{(n)}$, $\beta$ does indeed belong to $C^{(n)}$. 

Every $E_0$ cardinal is, evidently,  $C^{(1)}$-$E_0$. However, a simple reflection argument shows that the least $C^{(n)}$-$E_0$ cardinal, for $n\geq 1$, is smaller than the first cardinal in $C^{(n+1)}$, and therefore it is not  $C^{(n+1)}$-$E_0$. For suppose  $\alpha \in C^{(n+1)}$ is less than or equal to the first $C^{(n)}$-$E_0$ cardinal $\kappa$. Then  $V_{\alpha}$ satisfies the  $\Sigma_{n+1}$ statement asserting the existence of a $C^{(n)}$-$E_0$ cardinal, because $\kappa$ witnesses it in $V$.
But if $V_\alpha$ thinks some $\lambda$ is a $C^{(n)}$-$E_0$ cardinal, then so does  $V$, contradicting the minimality of  $\kappa$.

\begin{proposition}
\label{prop7.1}
If $\kappa$ is $C^{(n)}$-$E_0$, then it  is $C^{(n)}$-$m$-huge, for all $m$, and there is a $\kappa$-complete normal ultrafilter $\mathcal{U}$ over $\kappa$ such that 
$$\{\alpha <\kappa:\alpha\mbox{ is $C^{(n)}$-$m$-huge for every }m\}\in \mathcal{U}.$$
\end{proposition}

\begin{proof}
Let $j:V_\delta \to V_\delta$ witness that $\kappa$ is $C^{(n)}$-$E_0$, with $\delta$ limit. Then as in \cite{K}, 24.8 one can show that the ultrafilter  $\mathcal{V}$ over $\mathcal{P}(\lambda)$, where $\lambda =j^m(\kappa)$, defined by
$$X\in \mathcal{V}\mbox{ if and only if }j''\lambda \in j(X)$$
witnesses that $\kappa$ is $C^{(n)}$-$m$-huge. Let $\mathcal{U}$ be the usual $\kappa$-complete normal ultrafilter over $\kappa$ obtained from $j$.  Since $\mathcal{V}\in V_\delta$, $V_\delta$ satisfies that $\kappa$ is $C^{(n)}$-$m$-huge, and so $\{\alpha <\kappa:\alpha\mbox{ is $C^{(n)}$-$m$-huge for every }m\}\in \mathcal{U}$.
\end{proof}

\begin{proposition}
\label{propI3}
Suppose $j:V_{\delta}\to V_{\delta}$ witnesses that $\kappa$ is $E_0$, and $\delta$ is a limit ordinal. Then the following are equivalent for every $n\geq 1$,
\begin{enumerate}
\item $j^m(\kappa)\in C^{(n)}$, all $1\leq m <\omega$.
\item $\delta\in C^{(n)}$.
\end{enumerate}
\end{proposition}

\begin{proof}

(1) implies (2) is immediate, since $\delta =sup\{ j^m(\kappa):m<\omega \}$.
And (2) implies (1)  directly follows from the easily verifiable fact that $V_\kappa$ and $V_{j^m(\kappa)}$, all $m\geq 1$, are elementary substructures of $V_\delta$.
\end{proof}


This suggests the following definitions.

\begin{definition}
We say that $\kappa$ is \emph{$m$-$C^{(n)}$-$E_0$}, where $m\geq 1$,  if it is $C^{(n)}$-$E_0$, witnessed by some $j:V_\delta \to V_\delta$ with $j^{m'}(\kappa)\in C^{(n)}$ for all $1\leq m'\leq m$. And we say that $\kappa$ is \emph{$\omega$-$C^{(n)}$-$E_0$} if it is $C^{(n)}$-$E_0$, witnessed by $j:V_{\delta}\to V_{\delta}$ with $\delta \in C^{(n)}$. 
\end{definition}

Clearly, $\kappa$ is $E_0$ if and only if it is $C^{(1)}$-$E_0$ if and only if it is $\omega$-$C^{(1)}$-$E_0$.

Observe that if $\kappa$ is $m$-$C^{(n)}$-$E_0$,  where $1\leq m <\omega$, witnessed by $j:V_\delta \to V_\delta$ with $\delta$ the minimal such, then $\delta \not \in C^{(2)}$. Otherwise, $V_\delta$ would reflect the $\Sigma_2$ statement:
$$\exists \eta \exists k(k:V_\eta \to V_\eta \wedge crit(k)=\kappa \wedge \bigwedge_{1\leq  m'\leq m} (k^{m'}(\kappa)=j^{m'}(\kappa)))$$
where $\kappa$ and $j^{m'}(\kappa)$, all $1\leq m' \leq m$, are parameters, and so a witness to the statement would yield a counterexample to the minimality of $\delta$. It follows that $j$ cannot witness that $\kappa$ is $\omega$-$C^{(2)}$-$E_0$.

Note also, by the reflection argument given prior to Proposition \ref{prop7.1},  that the least $\omega$-$C^{(n)}$-$E_0$ cardinal $\kappa$ is  smaller than the first cardinal in   $C^{(n+1)}$. Hence, no cardinal less than or equal to $\kappa$ is $C^{(n+1)}$-$E_0$, for $n\geq 1$.

\begin{proposition}
\label{mlevelI3}
The least $m$-$C^{(n)}$-$E_0$ cardinal is not $(m+1)$-$C^{(n)}$-$E_0$, for all $m\geq 1$ and $n\geq 2$.
\end{proposition}

\begin{proof}
Suppose $\kappa$ is the least $m$-$C^{(n)}$-$E_0$ cardinal and suppose, aiming for a contradiction, that $j:V_\delta \to V_\delta$ is elementary, with $crit(j)=\kappa$ and $j^{m'}(\kappa)\in C^{(n)}$, for all $m'\leq m+1$.
Then, $V_{j^{m+1}(\kappa)}$ satisfies the sentence:
$$\exists i \, \exists \beta \, \exists \mu \, (i:V_\beta \to  V_\beta \mbox{ is elementary}\wedge crit(i)=\mu \wedge \mu <j (\kappa)  \, \wedge$$
$$ \bigwedge_{1\leq m' \leq m} i^{m'}(\mu)=j^{m'}(\kappa))$$
because $j$, $\delta$, and $\kappa$ witness it and the sentence is $\Sigma_2$ with parameters $j^{m'}(\kappa)$, all $1\leq m'\leq m$. Hence, by elementarity the following holds in $V_{j^m(\kappa)}$:$$\exists i \, \exists \beta \, \exists \mu \, (i:V_\beta \to  V_\beta \mbox{ is elementary}\wedge crit(i)=\mu \wedge \mu <\kappa  \, \wedge$$
$$ \bigwedge_{1\leq m' \leq m} i^{m'}(\mu)=j^{m'-1}(\kappa))$$
where $j^0(\kappa)=\kappa$. And since $j^m(\kappa)\in C^{(n)}$, it also holds  in $V$. 
But if  $\mu$ witnesses it, then $\mu$  is $m$-$C^{(n)}$-$E_0$, contradicting the minimality of $\kappa$.
\end{proof}

\medskip

It is not hard to see that $\kappa$ is $E_0$, witnessed by an elementary embedding $j:V_\delta \to V_\delta$, if and only if $\kappa$ is the critical point of an  embedding $k:V_{\delta +1}\to V_{\delta +1}$ which is $\Sigma_0$ elementary, i.e., it preserves truth for bounded formulas, with parameters. The main point is to notice that $j$ extends uniquely to a $\Sigma_0$ elementary embedding $k:V_{\delta +1}\to V_{\delta+1}$ by letting $k(A):=\bigcup_{\alpha <\delta}j(A\cap V_\alpha)$, for all $A\subseteq V_\delta$ (see \cite{L} or \cite{D} for details).

So it is only natural to consider the principles $E_i$, for $1\leq i\leq \omega$ (\cite{L}), which assert the existence of  a non-trivial $\Sigma_i$ elementary embedding $j:V_{\delta +1}\to V_{\delta +1}$, i.e., $j$ preserves the truth of $\Sigma_i$ formulas, with parameters. Thus, $E_\omega$ asserts that $j$ is fully elementary. 
$E_1$ and $E_\omega$ are also known in the literature as I2 and I1, respectively (see \cite{K}).

Observe that an embedding $j:V_{\delta +1}\to V_{\delta +1}$ is $\Sigma_i$ elementary if and only if its restriction to $V_\delta$ is $\Sigma^1_i$ elementary. (Recall that a formula is $\Sigma^1_i$ if it is a second order formula which begins with exactly $i$-many  alternating second-order quantifiers, beginning with an existential one, and the rest of the formula has only first-order quantifiers.) We shall later make use of the fact (folklore) that for each $i\geq 1$, the formula 
$$j:V_\delta \to V_\delta \mbox{ is } \Sigma^1_i\mbox{ elementary}$$
is $\Pi_{i+1}$ expressible in $V_{\delta +1}$, in the parameters $j$ and $\delta$, for it is equivalent to: 
For every 
$A\subseteq V_\delta$ and every $\Sigma^1_i$ formula $\exists X_1 \forall X_2\ldots \exists X_i  \psi (X_1,  \ldots ,X_i,Y)$, where $\psi$ has only first-order quantifiers, 
$$\exists X_1 \forall X_2 \ldots \exists X_i  (\langle V_\delta ,\in , X_1,\ldots ,X_i, A\rangle  \models \psi (X_1, \ldots ,X_i, A))$$
$$\mbox{if and only if}$$
$$\exists X_1 \forall X_2 \ldots \exists X_i (\langle V_\delta , \in , X_1,\ldots , X_i, A, j\rangle \models \psi (X_1, \ldots ,X_i, \bigcup_{\alpha <\delta} j(A\cap V_\alpha))).$$

Now, it was shown by Donald Martin that, for $i$ odd, if  $j:V_{\delta +1}\to V_{\delta +1}$ is $\Sigma_i$ elementary, then it is also $\Sigma_{i+1}$ elementary  (see \cite{L}). So one only needs to consider the principles $E_i$ when $i$ is even.

\medskip

Analogously as in the case of $E_0$, let us  call the critical point $\kappa$ 
of a $\Sigma_i$ elementary embedding $j:V_{\delta +1}\to V_{\delta +1}$ an \emph{$E_i$ cardinal}. If, in addition, $j(\kappa)\in C^{(n)}$, then we call $\kappa$  a \emph{ $C^{(n)}$-$E_i$ cardinal}. More generally, if $j^{m'}(\kappa)\in C^{(n)}$, for all $1\leq m'\leq m$, then we say that $\kappa$ is an \emph{$m$-$C^{(n)}$-$E_i$ cardinal}. And if $\delta \in C^{(n)}$, then we say that $\kappa$ is $\omega$-$C^{(n)}$-$E_i$. (Note that, by Proposition \ref{propI3}, $\delta \in C^{(n)}$ if and only if $j^m(\kappa)\in C^{(n)}$ for all $m$.) 

%


For each  $i \leq  \omega$ and $m, n\geq 1$,  the  existence of an $m$-$C^{(n)}$-$E_i$ cardinal can be expressed as a $\Sigma_{n+1}$ statement, namely
$$\exists \delta \exists j\exists \kappa (j: V_{\delta+1}\to V_{\delta+1}\mbox{ is $\Sigma_i$ elementary}\wedge crit(j)=\kappa \; \wedge$$
$$\forall \;1\leq m'\leq m (j^{m'}(\kappa)\in C^{(n)})).$$
(Notice that in the case $m=\omega$, $j^m (\kappa)=\delta$.) A  reflection argument similar to the one given prior to Proposition \ref{prop7.1} now yields that the least $m$-$C^{(n)}$-$E_i$ cardinal is smaller than the first cardinal in  $C^{(n+1)}$, for all $m, i\leq \omega$ and $n\geq 1$, and therefore smaller than the least $C^{(n+1)}$-$E_0$ cardinal.

Trivially, if  $\kappa$ is $m$-$C^{(n)}$-$E_{i+1}$, then it is $m$-$C^{(n)}$-$E_{i}$. In the case $i=0$ much more is true:  arguing similarly as in \cite{K}, 24.4 one can show that there is a normal ultrafilter $\mathcal{U}$ over $\kappa$ such that the set of cardinals $\alpha <\kappa$ that are  $m$-$C^{(n)}$-$E_0$ belongs to $\mathcal{U}$.  In the general case we have the following.


\begin{theorem}
\label{lasttheo}
Suppose $j:V_{\delta +1}\to V_{\delta +1}$ witnesses that $\kappa$ is an $m$-$C^{(n)}$-$E_{i+2}$ cardinal, where $i,n<\omega$ and $m\leq \omega$. Then the set of $m$-$C^{(n)}$-$E_{i}$ cardinals is unbounded below $\kappa$.
\end{theorem}

\begin{proof}
Fix $\gamma <\kappa$.  Then the following holds in $V_{\delta +1}$:
$$\exists k\, \exists \beta \, \exists \alpha\,  (k:V_\beta \to V_\beta \mbox{ is }\Sigma^1_i \mbox{ elementary}\wedge \gamma < crit(k)=\alpha <j(\kappa)\,  \wedge$$
$$\forall \, 1\leq m'\leq m \, (k^{m'}(\alpha )=j^{m'}(\kappa)))$$
because $j$, $\delta$, and $\kappa$ witness it.

As we observed above, the formula ``$k:V_\beta \to V_\beta \mbox{ is }\Sigma^1_i \mbox{ elementary}"$ is $\Pi_{i+1}$ expressible in $V_{\delta +1}$ in the variables $k$ and $\beta$. So the last displayed statement is $\Sigma_{i+2}$, with $\gamma$ and  $\langle j(\kappa), \,j^2(\kappa),\ldots , \, j^m(\kappa)\rangle$ as parameters. Thus, since $j$ is $\Sigma_{i+2}$ elementary, $V_{\delta +1}$ satisfies:
$$\exists k\, \exists \beta \, \exists \alpha\,  (k:V_\beta \to V_\beta \mbox{ is }\Sigma^1_i \mbox{ elementary}\wedge \gamma <crit(k)=\alpha <\kappa\,  \wedge$$
$$\forall \, 1\leq m'\leq m\, (k^{m'}(\alpha )=j^{m'-1}(\kappa)))$$
where $j^0(\kappa)=\kappa$. If $k$, $\beta$, and $\alpha$ witness the statement, then the embedding $k:V_\beta \to V_\beta$ witnesses that $\alpha$ is an $m$-$C^{(n)}$-$E_{i}$ cardinal greater than $\gamma$.
\end{proof}

Similarly as in Proposition \ref{mlevelI3} one can show that the least $m$-$C^{(n)}$-$E_i$ cardinal, where $i\leq \omega$,   is not $(m+1)$-$C^{(n)}$-$E_i$, for all  $m\geq 1$, and all $n\geq 2$.

\medskip

To summarize, for every $n$, let $\kappa^{(n)}$ denote the first cardinal in $C^{(n)}$, and for $i\leq  \omega$ and $2\leq m \leq \omega$, let $\kappa_i^{(n)}$ and  $m$-$\kappa_i^{(n)}$ denote the least $C^{(n)}$-$E_i$ cardinal and  the least $m$-$C^{(n)}$-$E_i$ cardinal,  respectively. Then, assuming all these cardinals exist, we have:
$$\kappa^{(n)}<\kappa^{(n)}_i < m\mbox{-}\kappa^{(n)}_i < m\mbox{-}\kappa^{(n)}_{i+2} < m\mbox{-}\kappa^{(n)}_\omega < (m+1)\mbox{-}\kappa^{(n)}_\omega < \omega\mbox{-}\kappa^{(n)}_\omega < \kappa^{(n+1)},$$
for all $i$ in all cases; all $n$ in the case of the inequalities 1,3,4,and 6; all $n\geq 1$ in the case of the last inequality; all $n\geq 2$ in the case of inequalities 2 and 5; all $1\leq m \leq \omega$ in the case of inequalities 3,4, 5, 6, and 7; and all $2\leq m\leq \omega$ in the case of the second inequality.

The first inequality is clear. Inequalities 2 and 5 follow from an argument similar to the proof of Proposition \ref{mlevelI3}. Inequalities 3,4, and 6 follow from Theorem \ref{lasttheo}. And the last inequality can be shown by a reflection argument similar to that given just before Proposition \ref{prop7.1}.
%




%


%
\vskip 0.5cm

{\small

\noindent
Joan Bagaria,
ICREA (Instituci\'o Catalana de Recerca i Estudis Avan\c{c}ats) and
Departament de L\`ogica, Hist\`oria i Filosofia de la Ci\`encia, Universitat de Barcelona. \\
Montalegre 6,
08001 Barcelona, Catalonia (Spain). \\
{\tt joan.bagaria@icrea.cat}\\
{\tt bagaria@ub.edu}\\
{\tt http://www.icrea.cat/Web/ScientificStaff/Joan-Bagaria-i-Pigrau-119}

}


\begin{thebibliography}{99}


\bibitem{BCMR} Bagaria, J., Casacuberta, C.,  Mathias, A.R.D., and Rosick\'y, J. (2010). Definable orthogonality classes are small. Submitted for publication.


\bibitem{BDPT} Barbanel, J., Di Prisco, C. A., and Tan, I. B. (1984) Many times huge and superhuge cardinals. \textit{Journal of Symbolic Logic} {\bf 49}, 112 -- 122.

\bibitem{D} Dimonte, V. (2010). \textit{Non-Proper Elementary Embeddings beyond $L(V_{\lambda +1})$}. Doctoral dissertation. Universit\'a di Torino.



\bibitem{J2} Jech, T. (2003). \textit{Set Theory. The Third Millenium Edition, Revised and Expanded}.
Springer Monographs in Mathematics. Springer-Verlag, Berlin, Heidelberg.
\bibitem{K} Kanamori, A. (1994). \textit{The Higher Infinite: Large Cardinals in Set
Theory from Their Beginnings}. Perspectives in Mathematical Logic. Springer-Verlag,
Berlin, Heidelberg.

\bibitem{Ku} Kunen, K. (1971). Elementary embeddings and infinitary combinatorics.  \textit{Journal of Symbolic Logic} \textbf{36}, 407 -- 413.


\bibitem{L} Laver, R. (1997). Implications between strong large cardinal axioms.  \textit{Annals of Pure and Applied Logic} \textbf{90}, 79 -- 90.

\bibitem{M} Magidor, M. (1971) On the role of supercompact and extendible cardinals in logic. \textit{Israel Journal of Mathematics} \textbf{10}, 147 -- 157.

\bibitem{MS} Martin, D. A. and Steel, J. R. (1989) A proof of Projective Determinacy. \textit{Journal of the American Mathematical Society} Volume \textbf{2}, Number 1, 71 -- 125.


\end{thebibliography}
\end{document}